\documentclass[a4paper,10pt]{amsart}
\usepackage[utf8]{inputenc}
\usepackage{amsthm}
\usepackage{amsmath}
\usepackage{bm}
\usepackage{enumitem}
\usepackage{amsfonts}
\usepackage{amssymb}
\usepackage{mathtools}
\usepackage{appendix}
\usepackage{mathrsfs}
\usepackage{setspace}
\usepackage{comment}
\usepackage{xcolor}
\usepackage{todonotes}
\usepackage{thmtools} % solves problem of shared counters in statements and autoref
\usepackage[foot]{amsaddr}
\RequirePackage[a4paper,top=2.54cm,bottom=2.54cm,left=1.90cm,right=1.90cm,%
                headsep=1em,includehead,includefoot]{geometry}
\usepackage[pdfdisplaydoctitle,colorlinks,breaklinks,urlcolor=blue,linkcolor=blue,citecolor=blue]{hyperref} 
\usepackage[nameinlink,capitalise]{cleveref}%must always be the last
\newcommand{\C}{\mathbb{C}}

\renewcommand{\H}{\mathcal{H}}

\newcommand{\N}{\mathbb{N}}

\newcommand{\R}{\mathbb{R}}

\newcommand{\T}{\mathbb{T}}

\newcommand{\Z}{\mathbb{Z}}

\let\div\relax

\DeclareMathOperator{\div}{div}

\DeclareMathOperator{\grad}{\nabla}

\renewcommand{\epsilon}{\varepsilon}
\renewcommand{\setminus}{\smallsetminus}
\newcommand{\eps}{\epsilon}

\newcommand{\one}{\bm{1}}

\newcommand{\pa}[1]{\left(#1\right)}

\newcommand{\norm}[1]{\left\|#1\right\|}
\newcommand{\brak}[1]{\left\langle#1\right\rangle}

\newcommand{\expt}[2][]{\mathbb{E}_{#1}\left[#2\right]}
\newcommand{\Bt}[1]{{B^{n}_{#1}}}

\newtheorem{theorem}{Theorem}[section]
\newtheorem{definition}[theorem]{Definition}
\newtheorem{hypothesis}[theorem]{Hypothesis}
\newtheorem{corollary}[theorem]{Corollary}
\newtheorem{lemma}[theorem]{Lemma}
\newtheorem{proposition}[theorem]{Proposition}

\theoremstyle{remark}
\newtheorem{remark}[theorem]{Remark}

\numberwithin{equation}{section}
\newcommand{\sumkj}{\sum_{\substack{k\in \Z^{3}_0\\ j\in \{1,2\}}}}
\newcommand{\sumkmeanj}{\sum_{\substack{k\in \Z^{3}_0\\k_3=0,\quad j\in \{1,2\}}}}
\newcommand{\sumkmeanjuno}{\sum_{\substack{k\in \Z^{3}_0\\k_3=0,\quad j=1}}}
\newcommand{\sumkmeanjdue}{\sum_{\substack{k\in \Z^{3}_0\\k_3=0,\quad j=2}}}
\newcommand{\sumkmeancov}{\sum_{\substack{k\in \Z^{3}_0\\ k_3=0}}}
\newcommand{\sumkoscj}{\sum_{\substack{k\in \Z^{3}_0\\k_3\neq0,\quad j\in \{1,2\}}}}
\newcommand{\skj}{\sigma_{k,j}^{n}}
\newcommand{\skl}{\sigma_{k,l}^{n}}
\newcommand{\skjcomp}[1]{\sigma_{k,j}^{n,#1}}
\newcommand{\smkj}{\sigma_{-k,j}^{n}}

\newcommand{\smkjcomp}[1]{\sigma_{-k,j}^{n,#1}}
\newcommand{\sk}[1]{\sigma_{k,#1}^{n}}
\newcommand{\skcomp}[2]{\sigma_{k,#2}^{n,#1}}
\newcommand{\smk}[1]{\sigma_{-k,#1}^{n}}
\newcommand{\smkcomp}[2]{\sigma_{-k,#2}^{n,#1}}
\newcommand{\tkj}{\theta_{k,j}^{n}}
\newcommand{\tkl}{\theta_{k,l}^{n}}
\newcommand{\tmkj}{\theta_{-k,j}^{n}}
\newcommand{\tmkl}{\theta_{-k,l}^{n}}
\newcommand{\tkjcomp}[1]{\theta_{k,#1}^{n}}

\newcommand{\sumH}{\sum_{\substack{k_H\in \Z^2_0\\ n\leq \lvert k_H\rvert \leq 2n}}}

\newcommand{\sgn}{\text{sgn}}
\newenvironment{acknowledgements}{%
  \begin{abstract}
}{%
  \end{abstract}
}

\title[Mean-field Models as Scaling Limits of Stochastic Induction Equations]{Mean-field Magnetohydrodynamics Models as Scaling Limits of Stochastic Induction Equations}
\author[F. Butori]{Federico Butori}
\address{Scuola Normale Superiore, Piazza dei Cavalieri, 7, 56126 Pisa, Italia}
\email{\href{mailto:federico.butori at sns.it}{federico.butori at sns.it}}
\author[E. Luongo]{Eliseo Luongo}
\address{Scuola Normale Superiore, Piazza dei Cavalieri, 7, 56126 Pisa, Italia}
\email{\href{mailto:eliseo.luongo at sns.it}{eliseo.luongo at sns.it}}
\date\today
\keywords{Roberts Flow, Magnetohydrodynamics, Transport Noise, Eddy Viscosity Model, Alpha Effect, Beta Effect}
\subjclass{60H15, 76F25}
\date\today
\allowdisplaybreaks
\begin{document}
\begin{abstract}
We study the asymptotic properties of a stochastic model for the induction equations of the magnetic field in a three dimensional periodic domain. The turbulent velocity field driving the electromotive force on the magnetic field is modeled by a noise white in time. For this model we rigorously take a scaling limit leading to a deterministic model. While in case of isotropic turbulence this produces an additional dissipation in the limit model which influences also the decay rate of the Magnetic field in the stochastic model, the case of turbulence devoloped in a preferential direction allows us to find a dynamo effect.
\end{abstract}

\maketitle

\section{Introduction}\label{sec introduction}
The evolution of the magnetic field induced by the motion of a (non-relativistic) conducting fluid with resistivity $\eta$ (or conductivity $1/\eta$) is governed by the Maxwell equation and Ohm's law. Together, these laws allow to derive the \emph{Induction Equation}
\begin{align*}
\partial_t B_t& = \eta \Delta B_t + \nabla\times (u_t \times B_t), \qquad \div B_t=0.
\end{align*}
The two terms on the right hand side represent respectively the Ohmic dissipation of magnetic energy due to the resistivity of the media, and the magnetic field induced by the electromotive force $u_t\times B_t$, generated by the motion of the fluid.
This equation is then coupled with some equation from fluid dynamics describing the evolution of the velocity field $u_t$. A correct and complete description of the system then shall include also the back reaction of $B_t$ on $u_t$ resulting in a system of coupled equations. This set of equations is at the core of Magneto-Hydrodynamics (MHD) and at that of the kinematic and dynamical aspects of electrically conducting fluids. In this field, an aspect that has remained obscure for many years is the mechanism with which the turbulent motion of astrophysical fluids is able to sustain large magnetic fields for timescales much longer than the dissipative one, think for example to the Sun, Earth or galaxies. Whether the fluid is in the core of a planet or on the surface of a star, it is an almost universal fact that when it is involved in rotating motion, there is the appearance of a magnetic field. The capability of these system to self excite and sustain their magnetic fields was phenomenologically explained through Dynamo Theory. However, the derivation of a precise Dynamo Theory from the induction equation above, has eluded physicists for many years.  
It was not until the '60s, with the work of Steenback, Krause and Rädler, that the situation was substantially transformed, see the seminal paper \cite{steenbeck1966berechnung} or the books \cite{krause2016mean}, \cite{moffatt1978field} and the references therein. The development of their two scales method, and in turn of the mean-field description allowed for a systematic treatment of the dynamo problem for the first time. Their approach was based on splitting the equations and the two fields in large scales $\overline{B_t}, \overline{u_t}$ and small scales components $B'_t, \ u'_t$. When the large scales are produced by some type of averaging operation, taking the average of a product is not the same as taking the product of the averages, thus the equation for the large scales produces a coupling between large and small scales:
\begin{equation*}
    \partial_t \overline {B_t} = \eta \Delta \overline {B_t} + \nabla\times (\overline u_t \times \overline {B_t}) + \nabla\times  ( \overline{u'_t \times B'_t}).
\end{equation*}
The term $\overline{u'_t \times B'_t}$, often referred to as \emph{Turbulent electromotive force} describes the effect of the interaction of the small scales of $u$ and $B$ on the large scale components. In order to close the equation for the large scales, Krause and Rädler proposed as a first approximation to model this interaction by a mean-field effective term of the form
\begin{equation}\label{emf assumption}
    \mathcal{A}\overline{B_t}  - \eta_T \nabla\times \overline{B_t}.
\end{equation}
The derivation of this term and the quantities $\mathcal{A}$ and $\eta_T$ were performed under simplifying assumptions based on an empirical description of the turbulent velocity $u'$. The new two terms were called respectively \emph{alpha} and \emph{beta} effect terms. The mean field equation then takes the form 
\begin{equation*}
    \partial_t \overline {B_t} = (\eta+ \eta_T) \Delta \overline {B_t} + \nabla\times (\overline u_t \times \overline {B_t}) + \nabla\times  ( \mathcal A \overline{B_t}).
\end{equation*}
Under these assumptions, the authors concluded that the term $\eta_T \nabla\times \overline{B_t}$ is responsible for enhancing dissipation, while the term $\mathcal{A}\overline{B_t}$, being of the first order, can provide the mechanism necessary to sustain the magnetic field. In this setting, it is the relative magnitude of these two terms that in the end determines if the large scale magnetic field $\overline{B_t}$ is able to self-sustain or if it dissipates. For this reason, the precise derivation of the quantities $\eta_T$ and $\mathcal A$, is essential for applying this theory of mean-field electrodynamics, see \cite{brandenburg2012current} for a review of numerical studies on the dynamo problem. %Nowadays it is known that their derivation of the $alpha$ and $beta$ terms is not universally valid and that the alpha effect is not the only mechanisms that can induce dynamo activity.  We will get back on this topic at the end of the introduction.
Nowadays it is known that the exact derivation of the \emph{alpha} and \emph{beta} term in general is not an easy matter, and the theory of Krause and R\"adler alone is not always satisfactory, as it relies on an heuristic description of turbulence. We get back on this topic and on other possible mechanisms for dynamo activity at the end of the introduction.\\
In view of the discussion above, the aim of this work is twofold: on one hand we propose a mathematically rigorous and complete procedure to derive the kinematic mean field effective system of Krause and Rädler. We do so with a scaling limit of a stochastic induction equation, in which the turbulent velocity field is replaced by a random Gaussian field delta correlated in time and the product $u_t\times B_t$ is interpreted in Stratonovich form, according to a Wong-Zakai principle (\emph{cf.} \cite{debussche2022second} and the references therein). We refer to \autoref{main Theorem} for the precise convergence result. On the other hand, and as a byproduct of our scheme, we are able to directly link the \emph{alpha} and \emph{beta} terms to the covariance of the noise, thus to the statistics of the (Gaussian) turbulent field, so that we bypass the difficulty of having to deriving them, as they can be prescribed a priori by an appropriate choice of the noise. In this respect, the modeling choice of Gaussian noise can be seen as parallel to the second order correlation approximation (SOCA) of Krause and Rädler, \emph{cf. }\cite[Chapter 4]{krause2016mean}, in which the \emph{alpha} and \emph{beta} terms are deduced dropping correlations of order higher than two of $u'$. Ultimately, the Gaussian assumption, as well as the SOCA one, are too restrictive for the description of fully developed turbulence, and more realistic models should be taken into account, however, in the context of the scheme we adopt, this is still a challenging open problem. In order to stress the flexibility of our approach, we propose several choices of the parameters of noise that account for different turbulent regimes. We focus on flows with positive total Helicity, in the case of isotropic and anisotropic energy spectra (see \autoref{subsect description noise}). In particular, we are able to explicitly identify the role that the Helicity of $u_t$ has in the appearance of the effective \emph{alpha} term in our model and how it is naturally linked to its strength, in a way that is consistent with models of Mean Field MHD (see for instance chapters $3$ and $4$ in \cite{krause2016mean} and the last part of this introduction, for a discussion on the role of helicity). \\

\subsection*{The \emph{It\^o-Stratonovich diffusion limit}}
Our scheme is in spirit the same of our previous work \cite{butori2024stratonovich} and falls into the category of so called It\^o-Stratonovich diffusion limits initiated in \cite{galeati2020convergence}:
the It\^{o}-Stratonovich corrector appearing when we rewrite
the term $\nabla\times\left(  u_t\times B_{t}\right)  $ from the Stratonovich form to the It\^{o} one, corresponds precisely to the mean field term (or sum of terms) discovered in the
literature under the most natural closure formulae. But in our rigorous
derivation there is still the It\^{o} term, which should disappear in the scaling limit
limit. There is no obvious reason why it should
disappear. Indeed, previous results in such direction where limited to consider scalar and 2D models, where stretching of
vector quantities does not appear. In such case it is relatively easy to prove that the It\^{o}
term is negligible in the limit, see for example \cite{flandoli2021scaling}, \cite{flandoli2021quantitative}, \cite{carigi2023dissipation}, \cite{flandoli20232d}, \cite{flandoli2023reduced}, \cite{agresti2024anomalous}. In 3D, the problem was left open for a long time. Previous results in the 3D framework have drastic unphysical simplifactions as neglecting the stretching term \cite{FlLuo21}, or dealing with 2D3C models \cite{flandoli2023boussinesq} and their variants \cite{butori2024stratonovich}. This work, therefore, represents the first example where the techniques of the It\^o-Stratonovich diffusion limit are applied successfully in a completely 3D model. In particular it was conjectured that only a suitable geometry of the involved vector fields could remedy the excess of fluctuations in the stretching term \cite{flandoli2023boussinesq}. Our approach seems to go in the opposite direction, indeed we are able to consider random vector fields $u_t$ either in case of an isotropic structure of the covariance and in case of a preferential direction.\\ 
Although they are similar in spirit, from the physical standpoint one main differences arises between this work and \cite{butori2024stratonovich}: in our previous work, as expected by the well known no-dynamo theorems \cite[Section 3.5]{gilbert2003dynamo}, since our limit equations lose any dependence from one coordinate (becomes 2D3C), the \emph{alpha}-effect in the mean field equation cannot provide any dynamo activity for the slow-varying large scales, thus it might be classified as perturbed version of a no-dynamo theorem, displaying that a `too simple', almost two dimensional, turbulent motion of the fluid cannot sustain a dynamo (see instead \cite{roberts1972dynamo} for the study of the dynamo problem in 3D flows independent by one component and \cite{proctor1994lectures} for some others paradigmatic examples of dynamos). 
Here instead, even when we deduce weaker convergence results, these are enough to provide dynamo activity in the mean field equations arising from the \emph{alpha}-effect if the covariance of the noise is sufficiently anisotropic \emph{cf.} \autoref{dynamo_effect}. On the contrary, if the breaking of the symmetry of the covariance is too weak, the \emph{beta}-term is dominant, providing an enhanced decaying rate \emph{cf.} \autoref{beta_effect}. In particular, the dynamo activity in our model, arising by stochastic perturbations, is a peculiarity of the anisotropic 3D features of the noise, and has no counterpart in 2D models. 
From the mathematical standpoint instead, the regularizing effect of the thin domain is here lost, resulting in either convergence in weaker norms or in more restrictive choices of the noise coefficients. It is however important to keep in mind that, even though the choice of the noise might seems tailored to the needs of the mathematical treatise of the equation, it actually stems from physical considerations, and is tuned as to keep non-trivial and non-vanishing mean helicity of the flow throughout the limiting procedure.
\subsection*{Helicity, Dynamos and Second Order effects}
Let us mention for completeness, that helicity is not the only known mechanism possible for dynamo activity. After the works of Krause, R\"adler and Steenback, the ideas of scale separation and mean field effects were treated sistematically, relying only on homogenization techniques, by several authors, \cite{vishikI}, \cite{lanotte1999large}, \cite{frisch1991}, see for instance the review monograph \cite{zheligovsky2011large}. In this research line, the authors developed an asymptotic theory to study the linear and weakly nonlinear stability problem of MHD models. These techniques allow to derive more general expressions for the \textit{alpha} and \textit{beta} terms. In particular, these results reveal that the direct dependence of the \textit{alpha} coefficient on the mean helicity is not general, and that the problem of determining the alpha effect is instead quite delicate. This fact has been critically assessed for instance in \cite{alpha-effect_chaotic_flows}, \cite{Gilbert01061988}, \cite{Brandenburg2008}, where \textit{alpha}-effect dynamos where constructed for flows with zero mean helicity ($\langle u, \nabla \times u\rangle=0$) and even zero helicity spectrum (in Fourier coefficients, i.e. $(ik \times  \hat u_k )\cdot \hat u_k^*=0 \ \forall k$), yet non-zero helicity density ($u(x) \cdot\nabla \times u(x)$). These constructions rely on some ad hoc geometric features of the flow (see the discussion at the end of \cite{Gilbert01061988}), and show that the identification of a general form for the alpha term is not an easy matter, when the underlying flow is complex. By contrast, our flow is delta correlated in time and becomes delta correlated in space in the limit $n\rightarrow \infty$ (restoring somehow the validity of the assumptions of the models of Krause and R\"adler, \cite[Chapter 4]{molokov2007magnetohydrodynamics}), thus we are able to detect only helicity-type of alpha effects from the limit equation.
Finally, even in situations where the driving flow is parity invariant (that is invariant under simultaneous reflection of coordinates and velocity) and the \textit{alpha} term vanishes, dynamos could still take place. This is because second order effects (like the \textit{beta}-effect) become relevant. The development of negative eddy viscosities, that is the appearance of a `negative' beta term, is now a corroborated phenomenon, which is known to produce dynamos (even if with smaller growth rates than helical dynamos), as investigated in \cite{lanotte1999large}, \cite{frisch1991}. Without claiming any rigorous result concerning negative eddy viscosities let us mention a possible links with our findings. In our model, we do not see such phenomena in the limit mean field equations ($\eta_T$ is strictly positive or zero, depending on the regime), however, we could be witnessing something similar in the stochastic model, that is at finite $n$. Let us explain what we mean. It is known \cite{lanotte1999large} that the \textit{beta} terms is certainly positive if the molecular diffusivity is large enough, as it is expected to dominate turbulent effects. However, if the magnetic Reynolds number is increased, the turbulent effects may become dominant and overwhelm the stabilizing effect of diffusion. By the analysis that we carried in this work, we noticed a similar behavior in the stochastic system. Namely, we are able to close our main energy estimates \autoref{prop:compact_space}, only if we have a large enough diffusivity $\eta$ to absorb some terms coming from the contribution of the stochastic stretching. This is expressed in the body of the paper in the reverse way, that is, the strength of the noise must be bounded by a quantity inversely proportional to $\eta$, as expressed, in the various regimes, by the hypothesis on the noise in \autoref{subsect description noise}. This suggesting that if we do not balance the noise with large enough diffusivity $\eta$, this might destroy the regularizing properties of the mean field $beta$-term, and in turn, that our limitation on the size of the eddy viscosity (which is proportional to the strength of the noise) in order to prove the validity of the mean field model, \emph{cf.} \autoref{dimension turbulent viscosity}, is not only technical but structural, and it is what we need to rule out the onset of negative eddy viscosity effects. We discuss this fact and its connection to our analysis more in details in \autoref{rmk negative eddy}. 

\subsection*{Content of the paper}
The content of the paper is the following. In \autoref{sec functional analysis main results} we introduce rigorously our framework. In particular, we define the structure of the noise in \autoref{subsect description noise}, while we state our main result and its consequences in \autoref{sec main result}. The remainder of the paper is dedicated to proving \autoref{main Theorem}. In \autoref{section a priori estimates} we prove some estimates for $B_t$ uniformly in the scaling limit of our noise, then we show the required convergence in \autoref{sec proof main thm}.

\section{Functional Setting and Main Results}\label{sec functional analysis main results}
In \autoref{notation sec} we introduce the notation we follow along the paper. In \autoref{preliminaries subsec} we introduce some functional analysis tools we employ in order to carry on rigorously our analysis described in \autoref{sec introduction}. We describe the noise we employ in \autoref{subsect description noise} providing some insights in the meaning of the assumptions and comparisons with the existing literature. Lastly \autoref{sec main result} is devoted to present our main results. 
\subsection{Notation}\label{notation sec}
In the following we denote by $\mathbb{T}^3=\mathbb{R}^3/(2\pi\mathbb{Z})^3$ the three dimensional torus and by $\Z^3_0=\Z^3\setminus \{0\}$ (resp. $\Z^2_0=\Z^2\setminus \{0\}$) the three (resp. two) dimensional lattice made by points with integer coordinates. We introduce a partition of $\Z^{3}_0=\Gamma_+\cup \Gamma_-$ such that 
\begin{align*}
    \Gamma_+=\{k\in \Z^{3}_0: k_3>0 \vee (k_3=0 \land k_2>0) \vee (k_3=k_2=0 \land k_1>0)\},\quad \Gamma_{-}=-\Gamma_+.
\end{align*}\\ 
We denote by $\{e_1,e_2,e_3\}$ the canonical basis of $\R^3$. If $v=(v^1,v^2,v^3)^t$ and $w=(w^1,w^2,w^3)^t$ are two vector fields we will denote by $v^H=(v^1,v^2)^t$ and $\mathcal{L}_v w$ for the Lie derivative, i.e. $\mathcal{L}_v w=v\cdot\nabla w -w\cdot\nabla v.$\\
For each $j\geq 0,\ l>1$ we denote by 
\begin{align*}
    \zeta^n_{s,j}&= n^{j-3}\sum_{\substack{k\in \Z^3_0,\\ n\leq \lvert k\rvert \leq 2n}}\frac{1}{\lvert k\rvert^{j}},\quad \zeta_{s,j}=4\pi \begin{cases}
     \log 2 & \text{ if } j=3,\\
        \frac{2^{3-j}-1}{3-j} & \text{ if } j\neq 3    
    \end{cases},\\
    \quad
    \zeta^n_{H,j}&= n^{j-2}\sumH \frac{1}{\lvert k_H\rvert^j},\quad \zeta_{H,j}=\int_{1\leq\lvert x\rvert\leq 2}\frac{dx_1dx_2}{\lvert x\rvert^j}=2\pi \begin{cases}
        \log 2 & \text{ if } j=2,\\
        \frac{2^{2-j}-1}{2-j} & \text{ if } j\neq 2
    \end{cases}.
\end{align*}
In particular, by the approximation properties of Riemann sums, for each $j\geq 0$ it holds
\begin{align*}
 \zeta_{s,j}^n= \zeta_{s,j}+O(n^{-1}),\quad \zeta_{H,j}^n= \zeta_{H,j}+O(n^{-1}).
\end{align*}
Moreover, we introduce the following function of the positive octant 
\begin{align*}
\chi:\R_+\times \R_+\times \R_+ \rightarrow \R,\quad \chi(\alpha,\beta,\gamma)=\begin{cases}
    1\quad &\text{if } \alpha=\beta=\gamma=3;\\
    (\gamma-3)\wedge 1 \quad &\text{if } \alpha=2,\ \beta=4,\ \gamma>3;\\
     (\alpha-2)\wedge(\beta-2)\wedge(\gamma-3)\quad &\text{if } \alpha>2,\ \beta>2,\ \gamma>3;\\
    0 \quad &\text{otherwise}.
\end{cases}    
\end{align*}
We denote by $[M, N]$ the quadratic covariation process,
which is defined for any couple $M,\ N$ of square integrable real semimartingales and we extend it by bilinearity to the analogue complex valued processes.\\
Let $a,b$ be two positive numbers, then we write $a\lesssim b$ if there exists a positive constant $C$ such that $a \leq C b$ and $a\lesssim_\xi b$ when we want to highlight the dependence of the constant $C$ on a parameter $\xi$. If $x\in \C$ we denote by $\overline{x}$ its complex conjugate. If $\mathcal{Y}$ is a second-order tensor we denote by $\norm{\mathcal{Y}}_{HS}$ its Hilbert-Schmidt norm. Lastly, if $x\in \R^2$ we denote by $x^{\perp}=(-x_2,x_1)^t.$
\subsection{Preliminaries}\label{preliminaries subsec}
Let us start setting some classical notation and properties of operators in the periodic setting before describing the main contributions of this work. We refer to monographs \cite{Engel_Nagel, FlaLuoWaseda, Marchioro1994,pazy2012semigroups,temam1995navier,temam2001navier, trie1995fun} for a complete discussion on the topics shortly recalled in this section.
\subsubsection{Function Spaces in the Periodic Setting}\label{function spaces periodic}
 Let $ \left(H^{s,p}(\T^3),\norm{\cdot}_{H^{s,p}}\right),\ s\in\mathbb{R},\ p\in (1,+\infty)$ be the Bessel spaces of periodic functions. In case of $p=2$, we simply write $H^{s}(\T^3)$ in place of $H^{s,2}(\T^3)$ and we denote by $\langle \cdot,\cdot\rangle_{H^s}$ the corresponding scalar products.
In case of $s=0$, we write $L^2(\T^3)$ instead of $H^0(\T^3)$ and we neglect the subscript in the notation for the norm and the inner product. With some abuse of notation, for $s>0$, we denote the duality between $H^{-s}$ and $H^s$ by $\langle\cdot,\cdot\rangle$.\\
We denote by $\Dot{H}^{s,p}(\T^3)$ (resp. $\Dot{H}^{s}(\T^3),\ \Dot{L^2}(\T^3)$) the closed subspace of $H^{s,p}(\T^3)$ (resp. $H^{s}(\T^3),\ L^2(\T^3)$) made by zero mean functions with norm inducted by $H^{s,p}(\T^3)$ (resp. $H^{s}(\T^3),\ L^2(\T^3)$) 
and by $Q$ the orthogonal projection of $L^2(\T^3)$ on $\Dot{L}^2(\T^3)$. The restriction (resp. extension) of $Q$ to $H^s(\T^3)$ (resp. $H^{-s}(\T^3)$) for $s>0$ is the orthogonal projection on $\Dot H^s(\T^3)$ (resp. $\Dot H^{-s}(\T^3)$). A different characterization of the fractional Sobolev spaces can be given in terms of powers of the Laplacian. Denoting by \begin{align*}
    \Delta: \mathsf{D}(\Delta)\subseteq \Dot{L^2}(\T^3)\rightarrow \Dot{L}^2(\T^3),
\end{align*}
where $\mathsf{D}(\Delta)=\Dot{H}^2(\T^3)$, it is well known that ${\Delta}$ is the infinitesimal generator of analytic semigroup of negative type that we denote by $e^{t\Delta}:\Dot L^2(\T^3)\rightarrow \Dot L^2(\T^3)$ and moreover for each $s>$0, $\mathsf{D}((-\Delta)^{s})$ (resp. $\left(\mathsf{D}((-\Delta)^{s})\right)^{\prime}$) can be identified with $\Dot H^{2s}(\T^3)$ (resp. $\Dot H^{-2s}(\T^3)$).\\
Similarly, we introduce the Bessel spaces of vector fields
\begin{align*}
    {H}^{s,p}(\T^3;\mathbb{R}^3)&= \{u=(u^1,u^2, u^3)^t:\ u^1,u^2, u^3\in H^{s,p}(\T^3)\},\\  \langle u, v\rangle_{{H}^s}&=\sum_{j=1}^3\langle u^j,v^j\rangle_{H^s}, \quad \text{for } s\in \mathbb{R}.
\end{align*}
Again, in case of $s=0$ we write ${L}^2(\T^3;\R^3)$ instead of ${H}^0(\T^3;\R^3)$ and we neglect the subscript in the notation for the norm and the scalar product. Similarly we denote by $\Dot H^s(\T^3;\R^3)$ (resp. $\Dot L^2(\T^3;\R^3)$) the closed subspace of $H^s(\T^3;\R^3)$ (resp. $L^2(\T^3;\R^3)$) made by zero mean functions. With some abuse of notation we usually employ the same symbols that we introduced for norms and inner products of scalar functions also for the corresponding vector fields and we continue to denote the orthogonal projection between $H^s(\T^3;\R^3)$ and $\Dot{H}^s(\T^3;\R^3),\ s\in \R$, by $Q$, since it acts component-wise.

Lastly we denote by $\mathbf{H}^{s}$ (resp. $\mathbf{L}^2$) the closed subspace of $\Dot{H}^s(\T^3;\R^3)$ (resp. $\Dot{L}^2(\T^3;\R^3)$) made by zero mean, divergence free vector fields with norm induces by $H^s(\T^3;\R^3)$ (resp. $L^2(\T^3;\R^3)$) and by \begin{align*}
    P:\Dot L^2(\T^3;\R^3)\rightarrow \mathbf{L}^2
\end{align*} 
the Leray projection which, in the periodic setting, is an orthogonal projection even when acting between $\Dot H^s(\T^{3};\R^3)$ and $\mathbf{H}^s.$ If $\{\frac{k}{\lvert k\rvert}, a_{k,1}, a_{k,2}\}$ is an orthonormal system of $ \R^3$ for each $k\in\Z^3_0$ then, it is well-known that the family $\{\frac{1}{(2\pi)^{3/2}}a_{k,j}e^{ik\cdot x}\}_{\substack{k\in \Z^3_0\\
j\in \{1,2\}}}$ is an orthonormal system of $\mathbf{L}^2$ and is orthogonal in $\mathbf{H}^s$ for each $s\in \R$.\\
We conclude this subsection introducing some classical notation when dealing with stochastic processes taking values in separable Hilbert spaces. Let $Z$ be a separable Hilbert space, with associated norm $\| \cdot\|_{Z}$. We denote by $C_{\mathcal{F}}\left(  \left[  0,T\right]  ;Z\right) $ the space of weakly continuous adapted processes $\left(  X_{t}\right)  _{t\in\left[
0,T\right]  }$ with values in $Z$ such that
\[
\mathbb{E} \bigg[ \sup_{t\in\left[  0,T\right]  }\left\Vert X_{t}\right\Vert
_{Z}^{2}\bigg]  <\infty
\]
and by $L_{\mathcal{F}}^{p}\left(  0,T;Z\right),\ p\in [1,\infty),$ the space of progressively
measurable processes $\left(  X_{t}\right)  _{t\in\left[  0,T\right]  }$ with
values in $Z$ such that
\[
    \mathbb{E} \bigg[ \int_{0}^{T}\left\Vert X_{t}\right\Vert _{Z}^{p}dt \bigg]
<\infty.
\]
\subsubsection{Operators in the Periodic Setting}\label{operators_periodic_setting}
In this section we are interested to the properties of differential operators of the form
\begin{align*}
\mu\partial_{11}+\mu\partial_{22}+\nu\partial_{33},\ \mu\geq \nu>0,
\end{align*}
the case $\mu=\nu$ corresponding to the Laplacian. It is well known that for each $\mu\geq \nu>0$, the operator
\begin{align*}
    A_{\mu,\nu}:\Dot{H}^2(\T^3)\subseteq \Dot{L}^2(\T^3)\rightarrow \Dot{L}^2(\T^3),\quad A_{\mu,\nu}f=\mu\partial_{11}f+\mu\partial_{22}f+\nu\partial_{33}f 
\end{align*}
is closed and generates an analytic semigroup of negative type which we denote by $S^{\mu,\nu}(t):\Dot L^2(\T^3)\rightarrow \Dot L^2(\T^3)$. The family $\{\frac{1}{(2\pi)^{3/2}}e^{ik\cdot x}\}_{k\in \Z^{3}_0}$  is a complete orthogonal systems of $\Dot{H}^s(\T^3)$ made by eigengunctions of $A_{\mu,\nu}$ for each $\mu\geq\nu>0$. Moreover it is orthonormal in $\Dot L^2(\T^3)$. This allows us to write the operators $A_{\mu,\nu},\ S^{\mu,\nu}(t)$ in Fourier basis as \begin{align}\label{fourier_description_op}
A_{\mu,\nu}f=-\frac{1}{(2\pi)^3}\sum_{\substack{k\in \Z^3_0}}\left(\mu(k_1^2+k_2^2)+\nu k_3^2\right)\langle f, e^{-ik\cdot x}\rangle e^{ik\cdot x},\ S^{\mu,\nu}(t)=\frac{1}{(2\pi)^3}\sum_{\substack{k\in \Z^3_0}}e^{-\left(\mu(k_1^2+k_2^2)+\nu k_3^2\right)t}\langle g, e^{-ik\cdot x}\rangle e^{ik\cdot x}   
\end{align}
for $f\in \Dot{H}^2(\T^3),\ g\in \Dot{L}^2(\T^3).$
The expression in Fourier basis has as immediate consequences the following facts which we recall here for the convenience of the reader.
\begin{lemma}\label{properties_semigroup_splitting_extension}
For each $s\geq 0$, the extension (resp. restriction) of $S^{\mu,\nu}(t):\Dot{H}^{-s}(\T^3)\rightarrow \Dot{H}^{-s}(\T^3)$ (resp. $S^{\mu,\nu}(t):\Dot{H}^{s}(\T^3)\rightarrow \Dot{H}^{s}(\T^3)$) is an analytic semigroup of negative type. Moreover, for each $\mu_1\geq \nu_1>0,\ \mu_2\geq \nu_2>0$ the operators $S^{\mu_1,\nu_1}(t)$ and $S^{\mu_1,\nu_1}(t)$ commute. In particular
\begin{align*}
    S^{\mu_1,\nu_1}(t)S^{\mu_2,\nu_2}(t)=S^{\mu_2,\nu_2}(t)S^{\mu_1,\nu_1}(t)=S^{\mu_1+\mu_2,\nu_1+\nu_2}(t).
\end{align*}
\end{lemma}
\begin{lemma}\label{Properties semigroup}
Let $q\in \Dot{H}^{s}(\T^3),\ s\in \mathbb{R}$. Then:
\begin{enumerate}[label=\roman*)]
    \item for any $\varphi\geq 0,$ it holds $\lVert S^{\mu,\nu}(t)q\rVert_{\Dot{H}^{s+\varphi}}\lesssim_{\varphi} (t\nu)^{-\varphi/2}\lVert q\rVert_{\Dot{H}^{s}}$, the hidden constant being increasing in $\varphi$;
    \item for any $\varphi\in [0,2],$ it holds $\lVert \left( I-S^{\mu,\nu}(t)\right)q\rVert_{\Dot{H}^{s-\varphi}}\lesssim (t\mu)^{\varphi/2}\lVert q\rVert_{\Dot{H}^{s}}$. The hidden constant being uniformly bounded for $\varphi\in[0,2]$.
\end{enumerate}
\end{lemma}
Lastly we need a closedness results on operators of this form, which has some similarities with the second item of \autoref{Properties semigroup}. We postpone its simple proof to \autoref{app_convergence_operator} for the convenience of the reader.
\begin{lemma}\label{convergence_operators}
Let $q\in \Dot{H}^{s}(\T^3),\ s\in \mathbb{R}$. If $\mu_1\geq \nu_1>0$ and $\mu_2\geq \nu_2>0$,
 for any $\varphi\in [0,2],$ it holds \begin{align*}
     \lVert \left( S^{\mu_1,\nu_1}(t)-S^{\mu_2,\nu_2}(t)\right)q\rVert_{\Dot{H}^{s-\varphi}}\lesssim \left(\lvert \mu_1-\mu_2\rvert+\lvert \nu_1-\nu_2\rvert\right)^{\varphi/2}t^{\varphi/2}\lVert q\rVert_{\Dot{H}^{s}}. 
 \end{align*}
The hidden constant being uniformly bounded for $\varphi\in[0,2]$.
\end{lemma}
With some abuse of notation we denote by \begin{align*}
    A_{\mu,\nu}:\mathbf{H}^2\subseteq \mathbf{L}^2\rightarrow \mathbf{L}^2,\quad  A_{\mu,\nu} f=(A_{\mu,\nu} f^1,A_{\mu,\nu} f^2, A_{\mu,\nu} f^3)^t.
\end{align*}
The Fourier expansion of  \eqref{fourier_description_op}, then allows us to write for $f\in \mathbf{H}^s$
\begin{align*}
    A_{\mu,\nu}f=-\frac{1}{(2\pi)^3}\sumkj \left(\mu(k_1^2+k_2^2)+\nu k_3^2\right)\langle f, a_{k,j} e^{-ik\cdot x}\rangle e^{ik\cdot x}a_{k,j}. 
\end{align*}
This implies that $A_{\mu,\nu}$ generates an analytic semigroup on $\mathbf{L}^2$ which we still denote by $S^{\mu,\nu}(t)$ and \autoref{properties_semigroup_splitting_extension}, \autoref{Properties semigroup}, \autoref{convergence_operators} continue to hold in this vector valued framework.

\subsection{Description of the Stochastic System}\label{subsect description noise}
We begin with the description of the noise. We give a general description which cover all the three regimes we are interested in: we call them the Helical noise, the perturbed 2D noise and the Isotropic noise.
Choose $\rho\in [-1,1]$, a correlation factor whose role will be that of introducing some correlation between the vertical and horizontal components of the `2D' modes of the random field. Then assume  $\left(\Omega,\mathcal{F},\mathcal{F}_t,\mathbb{P}\right)$ is a filtered probability space such that $(\Omega, \mathcal{F},\mathbb{P})$ is a complete probability space, $(\mathcal{F}_t)_{t\in [0,T]}$ is a right continuous filtration and $\mathcal{F}_0$ contains every $\mathbb{P}$ null subset of $\Omega$.
Now, for each $n\in \N$, $\{W_t^{k,j}\}, \ {{k\in \Z^{3}_0, \  j\in \{1,2\}}}$ is a sequence of complex-valued Brownian motions adapted to $\mathcal{F}_t$ 
    such that
$W^{-k,j}_t=(W^{k,j}_t)^*$ and 
\begin{align*}
    \expt{W^{k,j}_1{W^{l,m}_1}^*}&=\begin{cases}
        2 & \textit{if } k=l, \ m=j\\
        2\rho &  \textit{if } k=l,\ k_3=0,\ m\neq j\\
        0 & \textit{otherwise};
    \end{cases}
    \quad \left[W^{k,j}_\cdot, W^{l,m}_\cdot\right]_t=\begin{cases}
        2t & \textit{if } k=-l, \ m=j\\
        2\rho t &  \textit{if } k=-l,\ k_3=0,\ m\neq j\\
        0 & \textit{otherwise}.    
        \end{cases}
\end{align*}
Then, for each $ k\in \Z^{3}_0,\ j\in\{1,2\}$ we denote by $\skj(x)=\tkj  a_{k,j}e^{ik\cdot x}$, where $\{\frac{k}{\lvert k\rvert}, a_{k,1}, a_{k,2}\}$ is an orthonormal system of $\R^3$ for $k_3> 0$ and $a_{k,j}=a_{-k,j}$ if $k_3<0$, while the (complex) coefficients $\theta_{k,j}^\eps$ will be defined later. 
It is well known that the family $\{a_{k,j}\frac{1}{2\pi\sqrt{2\pi}}e^{ik\cdot x}\}, \ {{k\in \Z^{3}_0, \ j\in\{1,2\}}}$ is a complete orthogonal systems of $\mathbf{H}^s$ made by eigenfunctions of $-\Delta$. Moreover it is orthonormal in $\mathbf{L}^2$. 
Without losing of generality, we make a choice that will simplify some of the computations: when $k_3=0$ we choose $a_{k,1}=(-k_2/\lvert k\rvert,k_1/\lvert k\rvert,0),\ a_{k,2}=(0,0,1)$ if $k\in \Gamma_{+}$, $a_{k,j}=a_{-k,j}$ otherwise. 
For a physical intuition behind our construction see our previous work \cite{butori2024stratonovich} and the discussion therein. 
Then we define \begin{align*}
    W^{n}_t=\sum_{\substack{k\in \Z^{3}_0\\ j\in \{1,2\}}}\skj W^{k,j}_t,
\end{align*} with $\skj=\tkj a_{k,j}e^{ik\cdot x}$.
With our choices, to ensure that $W^n$ takes values in a space of real valued functions, the coefficients must satisfy the condition \begin{equation}\label{symmetry of tks}
    {\tkj}^{*}=\theta_{-k,j}^n.\end{equation}
The explicit choices of the coefficients $\tkj$:
\begin{align}\label{definition thetakj}
    \tkj &=\one_{\{n\leq \lvert k\rvert\leq 2n\} }\begin{cases}
        \frac{i}{C_{1,H}\lvert k\rvert^{\alpha/2}}& \text{if } k_3=0, k\in\Gamma_{+},\ j=1\\
         \frac{-i}{C_{1,H}\lvert k\rvert^{\alpha/2}}& \text{if } k_3=0, k\in\Gamma_{-},\ j=1\\
        \frac{1}{C_{2,H}\lvert k\rvert^{\beta/2}}& \text{if } k_3=0,\ j=2\\
        \frac{1}{C_{V}\lvert k\rvert^{\gamma/2}} & \text{if } k_3\neq 0
    \end{cases}
\end{align}
 allows us to distinguish our three regimes:
\begin{hypothesis}[Helical]\label{HP noise 1} $ $
\begin{itemize}
    \item  $\alpha\geq 2,\ \beta>2,\ \gamma>3,\ \alpha+\beta=6$.
    \item  $C_{1,H},\ C_{2,H},\ C_V>0$ and $C_{1,H}>\sqrt{\frac{\zeta_{H,2}}{\eta}}$ if $\alpha=2$.
\end{itemize}
\end{hypothesis}
\begin{hypothesis}[Perturbed 2D]\label{HP noise 2} $ $
\begin{itemize}
    \item  $\alpha=2,\ \beta=4,\ \gamma\geq 5$.
    \item  $C_{1,H}>\sqrt{\frac{\zeta_{H,2}}{\eta}},\ C_{2,H},\ C_V>0$.
\end{itemize}
\end{hypothesis}
\begin{hypothesis}[Isotropic]\label{HP noise isotropo} $ $
\begin{itemize}
    \item  $\alpha=\beta=\gamma=3$.
    \item  $C_{1,H}=C_{2,H}=C_V>\sqrt{\frac{2\zeta_{s,3}}{3\eta}}$.
\end{itemize}
\end{hypothesis}
\begin{remark}
    The term `isotropic' in \autoref{HP noise isotropo} is understood with a slight abuse, indeed it refers to the fact that the scaling parameters are all equal, but not to the fact that the fields $W^n_t$ are actually isotropic, a condition that is met only in the special case of $\rho=0$, and if one sets all the coefficients to be real. However, since our analysis and our main result remains unchanged whether if $\rho$ is zero or not in this regime, we included also the latter case for generality. \\
    The Perturbed 2D regime, instead, is a sub case of the Helical regime. Its relevance comes from the fact that due to the strong suppression of true 3D modes ($\gamma \ge 5$), this model represents a small 3D perturbation of a 2D-3C flow with weak vertical components ($\beta = 4$), and akin to the works \cite{flandoli2023boussinesq}, \cite{butori2024stratonovich} which previously employed similar noises, this choice is the only one that allows us to obtain a slightly stronger convergence result.
\end{remark}
\begin{remark}\label{remark coefficients }
    Due to our choice of the coefficients $\tkj$ and the vectors $a_{k,j}$ it follows that for each $k$ such that $k_3=0,\ n\leq \lvert k\rvert \leq 2n$
    \begin{align*}
        \tkjcomp{1} a_{k,1}=\frac{i}{C_{1,H}}\frac{k^{\perp}}{\lvert k\rvert^{\alpha/2+1}},\ \tkjcomp{2} a_{k,2}=\frac{1}{C_{2,H}}\frac{e_3}{\lvert k\rvert^{\beta/2+1}}.
    \end{align*}
    \end{remark}
The space covariance function associated to our noise is \begin{align}\label{def covariance}
Q^{n}(x,y)=\expt{W^{n}_1(x) \otimes {W^{n}_1(y)}^* }=Q^{n}_0(x,y)+\overline{Q^{n}_{\rho}}(x,y)=\overline{Q^{n}}(x,y)+\left(Q^{n}\right)'(x,y)+\overline{Q^{n}_{\rho}}(x,y)
\end{align} 
where we denoted by \begin{align*}
\overline{Q^{n}}(x,y)&=2\sumkmeanj \left\lvert \tkj\right\rvert^2 a_{k,j}\otimes a_{k,j}e^{ik\cdot(x-y)},\quad \left(Q^{n}\right)'(x,y)=2\sumkoscj \left\lvert \tkj\right\rvert^2 a_{k,j}\otimes a_{k,j}e^{ik\cdot(x-y)},\\
\overline{Q^{n}_{\rho}}(x,y)&=2\rho \sumkmeancov  \left(\tkjcomp{1}\left(\tkjcomp{2}\right)^* a_{k,1}\otimes a_{k,2}+\tkjcomp{2}\left(\tkjcomp{1}\right)^* a_{k,2}\otimes a_{k,1}\right)e^{ik\cdot(x-y)}.   
\end{align*}
Due to our choice, we have that $\overline{Q^{n}}(x,y),$ $\left(Q^{n}\right)'(x,y)$ and $\overline{Q^{n}_{\rho}}(x,y)$ are all translation invariant, therefore we simply write $Q^{n}(x-y)$ (resp. $Q^{n}_0(x-y),\ \overline{Q^{n}}(x-y),\ \left(Q^{n}\right)'(x-y),\ \overline{Q^{n}_{\rho}}(x-y)$) in place of  $Q^{n}(x,y)$ (resp. $Q^{n}_0(x,y),\ \overline{Q^{n}}(x,y),\ \left(Q^{n}\right)'(x,y),\ \overline{Q^{n}_{\rho}}(x,y)$). In particular $Q^{n}_0(x),\ \overline{Q^{n}}(x),\ \left(Q^{n}\right)'(x)$ are mirror symmetric, namely it holds 
\begin{align}\label{mirror symmetric property 1}
Q^{n}_0(x)=Q^{n}_0(-x),\ \overline{Q^{n}}(x)=\overline{Q^{n}}(-x),\ \left(Q^{n}\right)'(x)=\left(Q^{n}\right)'(-x).  
\end{align}
On the contrary $\overline{Q^{n}_{\rho}}(x)$ is an odd function, in particular $\overline{Q^{n}_{\rho}}(0)=0$. As a consequence, our noise satisfies the mirror symmetry property if and only if $\rho=0$.

We recall here some properties of the covariance functions $\overline{Q^{n}_{0}}(x), (Q^{n}_0)'(x), \overline{Q^{n}_{\rho}}(x)$. Their proofs are a simple adaptation of the arguments of \cite[Section 2.3]{butori2024stratonovich} to which we refer for details.

\begin{lemma}\label{form covariance matrix non-isotropic}
If either \autoref{HP noise 1} or \autoref{HP noise 2} hold, then 
\begin{align*}
    \overline{Q^n_0}(0)=2\eta^n_T(e_1\otimes e_1+e_2\otimes e_2)+2\eta^n_R\left(e_3\otimes e_3\right),\quad (Q^{n}_0)'(0)=2(\eta^n_V-\eta^n_{R,V}) I+\eta^n_{R,V}I_2,
\end{align*}
where we set
\begin{align*}
    \eta^n_T&=\frac{\zeta_{H,\alpha}^n}{2C_{1,H}^2 n^{\alpha-2}}=O(n^{2-\alpha}),\quad \eta^n_R=\frac{\zeta^n_{H,\beta}}{C_{2,H}^2n^{\beta-2}}=O(n^{2-\beta}),\quad \eta^n_V=\frac{2\zeta^n_{s,\gamma-3}}{3C_V^2 n^{\gamma-3}}=O(n^{3-\gamma}),\\
    \eta^n_{R,V}&=\frac{\zeta^n_{H,\gamma}}{C_V^2 n^{\gamma-2}}=O(n^{2-\gamma}).
\end{align*}
In particular, if $\alpha=2$
\begin{align*}
    \eta^n_T=\frac{\zeta_{H,2}}{2C_{1,H}^2 }+O(n^{-1}).
\end{align*}
\end{lemma}
\begin{lemma}\label{form covariance matrix isotropic}
If \autoref{HP noise isotropo} holds, then 
\begin{align*}
    {Q^n_0}(0)=2\eta^n_{iso} I,
\end{align*}
where we set
\begin{align*}
    \eta^n_{iso}=\frac{2\zeta^n_{s,3}}{3C_V^2}=\frac{2\zeta_{s,3}}{3C_V^2}+O(n^{-1}).
\end{align*}
\end{lemma}
\begin{lemma}\label{lemma covariance operator 2}
It holds
\begin{align}\label{property derivatives 5}
    \left(\partial_{l} \overline{Q^{n}_0}\right)(0)&=2i\sum_{\substack{k\in \Z_0^3,\\
    k_3=0,\ j\in \{1,2\}}} \left\lvert\tkj\right\rvert^2 k_l a_{k,j}\otimes a_{k,j}=0,\quad \left(\partial_{l} {Q^{n}_0}\right)'(0)&=2i\sum_{\substack{k\in \Z_0^3,\\
    k_3\neq 0,\ j\in \{1,2\}}} \left\lvert\tkj\right\rvert^2 k_l a_{k,j}\otimes a_{k,j}=0.
\end{align}
As a consequence, for each $l,r,s\in \{1,2,3\}$ and $f\in L^2(\T^3),\ g\in H^1(\T^3)$ it holds
\begin{align}\label{mirror symmetry product 0}
    \sum_{\substack{k\in \Z_0^3,\\
    k_3=0,\ j\in \{1,2\}}}  \langle\partial_l g\skjcomp{l}, f\partial_r \smkjcomp{s}\rangle=\sum_{\substack{k\in \Z_0^3,\\
    k_3\neq 0,\ j\in \{1,2\}}}  \langle\partial_l g\skjcomp{l}, f\partial_r \smkjcomp{s}\rangle=0.
\end{align}
\end{lemma}
\begin{lemma}\label{convergence properties matrix}
We have \begin{align}\label{uniform bound step 1}
    \nabla\overline{Q^{n}_{\rho}}(0)=\frac{2\rho}{C_{1,H}C_{2,H}}\sumH \frac{1}{\lvert k_H\rvert^{2}}  \left(e_3 \otimes\frac{k_H^{\perp}}{\lvert k_H\rvert}\otimes  \frac{k_H}{\lvert k_H\rvert}-\frac{k_H^{\perp}}{\lvert k_H\rvert}\otimes e_3\otimes \frac{k_H}{\lvert k_H\rvert}\right).  
    \end{align}
In particular, it holds 
\begin{align*}
    \norm{\nabla\overline{Q^{n}_{\rho}}(0)}_{HS}\leq \frac{2\sqrt{2}\lvert\rho\rvert \zeta^n_{H,2}  }{C_{1,H}C_{2,H}}
\end{align*}
and
\begin{align*}
\{\mathcal{R}_{n}^{r,s}\}_{r,s\in\{1,2\}}&:=\partial_{s}\overline{Q^{n,3,r}_{\rho}}(0)= \frac{\rho \zeta^N_{H,2}} {C_{1,H}C_{2,H}}\begin{bmatrix} 
   0 & -1\\ 1 & 0   \end{bmatrix}=\frac{\rho \zeta_{H,2}} {C_{1,H}C_{2,H}}\begin{bmatrix} 
   0 & -1\\ 1 & 0   \end{bmatrix} +O(n^{-1}),\\
   & \{\mathcal{V}_{n}^{r,s}\}_{r,s\in\{1,2\}}:=\partial_{s}\overline{Q^{n,r,3}_{\rho}}(0)=-\partial_{s}\overline{Q^{n,3,r}_{\rho}}(0). 
\end{align*}
\end{lemma}
As discussed in \cite[Chapter 3]{FlaLuoWaseda}, the quantity $Q^n(0)$ is the core in order to obtain a nontrivial second order operator in our It\^o-Stratonovich diffusion limit, i.e. the \emph{beta} effect we discussed in \autoref{sec introduction}. In particular, the choice of \autoref{HP noise 1} and $\alpha=2$ or \autoref{HP noise isotropo} allows us to obtain a non trivial \emph{beta} effect in the scaling limit. As discussed in \autoref{sec introduction}, the \emph{alpha} effect is instead linked to the helicity of the velocity field. The role of the helicity is usually understood in terms of distortion of the axial magnetic field lines that generates currents in the rotating direction, which in turn induce magnetic fields in the axial direction, see for example \cite[Chapter 3, figure 3.2]{krause2016mean} and related comments. In our stochastic framework we denote the time average of the helicity density of our random velocity field as
\begin{align*}
    \H^n(x):= \frac{1}{2T}\expt{W^n_T(x)\cdot (\grad \times W^n_T(x))}.
\end{align*}
Exploiting the definition of our random vector field and arguing as in \cite[Appendix B]{butori2024stratonovich} we can easily see that $\H^n(x)$ is independent on $x$ and it satisfies
\begin{align*}
    \H^n(x)=-\frac{2\rho\zeta^{n}_{H,\frac{\alpha+\beta-2}{2}}}{C_{1,H}C_{2,H}n^{\frac{\alpha+\beta-6}{2}}}\longrightarrow \H:=-\frac{4\pi\rho\log 2}{C_{1,H}C_{2,H}},
\end{align*}
i.e. the choice of $\alpha+\beta=6$ is linked to reach a finite and nontrivial \emph{alpha} term in the limit. 
\subsection{Main Results}\label{sec main result}
Following the ideas introduced in \autoref{sec introduction}, we fix $T\in (0,+\infty)$ and we are interested in studying the properties of the following stochastic model on $\T^3$:
\begin{equation}\label{introductory equation stratonovich}
    \begin{cases}
        dB^{n}_{t}&=\eta \Delta B^{n}_{t}  dt + \sum_{\substack{k\in \Z^{3}_0\\ j\in \{1,2\}}}\mathcal{L}_{\skj}B^{n}_t  \circ dW^{k,j}_t, \\
        \operatorname{div}B^{n}_t&=0, \\
        B^{n}_t|_{t=0}&=B^{n}_0,
    \end{cases}
\end{equation}
with the coefficients $\skj$ defined in \autoref{subsect description noise}. We start rewriting it in It\^o form. First, let us observe that for each $k\in\Z^{3}_0,\ j\in \{1,2\}$ we have due to the linearity of $\mathcal{L}_{\skj}$
\begin{align*}
\mathcal{L}_{\skj}B^{n}_t  \circ dW^{k,j}_t&=\mathcal{L}_{\skj}B^{n}_t dW^{k,j}_t+\frac{1}{2}\left[ \mathcal{L}_{\skj}B^{n}_\cdot, W^{k,j}_\cdot\right]_t \\ & = \mathcal{L}_{\skj}B^{n}_t dW^{k,j}_t+\frac{1}{2}\mathcal{L}_{\skj}\left[ B^{n}_\cdot, W^{k,j}_\cdot\right]_t\\ & =\mathcal{L}_{\skj}B^{n}_t dW^{k,j}_t+\mathcal{L}_{\skj}\mathcal{L}_{\smkj}B^{n}_t dt+\rho\one_{\{k_3=0,\ l\neq j\}} \mathcal{L}_{\skj}\mathcal{L}_{\smk{l}}B^{n}_t dt.
\end{align*}
Therefore the equation for $B^{n}_t$ can be rewritten as
\begin{align*}
    \begin{cases}
        dB^{n}_{t}&=\left(\eta \Delta+\sumkj \mathcal{L}_{\skj}\mathcal{L}_{\smkj}+\rho\sumkmeancov \mathcal{L}_{\sk{1}}\mathcal{L}_{\smk{2}}+\mathcal{L}_{\sk{2}}\mathcal{L}_{\smk{1}} \right)  B^{n}_{t}  dt\\ & + \sum_{\substack{k\in \Z^{3}_0\\ j\in \{1,2\}}}\mathcal{L}_{\skj}B^{n}_t   dW^{k,j}_t, \\
        \operatorname{div}B^{n}_t&=0, \\
        B^{n}_t|_{t=0}&=B^{n}_0.
    \end{cases}    
\end{align*}
We are left to identify $\sumkj \mathcal{L}_{\skj}\mathcal{L}_{\smkj}+\rho\sumkmeancov \mathcal{L}_{\sk{1}}\mathcal{L}_{\smk{2}}+\mathcal{L}_{\sk{2}}\mathcal{L}_{\smk{1}} $.  In \autoref{app_ito_strat}, we sketch the proof of the following standard result for the convenience of the reader, in order to highlight the role of $Q^n_0(0)$ and $\nabla \overline{Q^n_{\rho}}(0)$ in our procedure.
\begin{lemma}\label{lemma ito strat corrector}
If $F:\T^3\rightarrow\R^3$ is a smooth, zero mean divergence free vector field, then it holds 
\begin{align*}
    \sumkj \mathcal{L}_{\skj}\mathcal{L}_{\smkj} F+\rho\sumkmeancov \left(\mathcal{L}_{\sk{1}}\mathcal{L}_{\smk{2}}+\mathcal{L}_{\sk{2}}\mathcal{L}_{\smk{1}}\right)F&=\Lambda^{n}F+ \Lambda_{\rho}^{n}F,\quad  
\end{align*}
where in the formula above we denoted by $\Lambda^{n}$ the differential operator 
\begin{align*}
\Lambda^{n}F=\begin{cases}
    (\eta^{n}_T+\eta^n_V-\eta^n_{R,V}/2)(\partial^2_{1}+\partial^2_{2})F+(\eta^{n}_R+\eta^n_V-\eta^n_{R,V})\partial^2_{3}F& \quad \text{ if \autoref{HP noise 1} holds},\\
    \eta^n_{iso}\Delta F & \quad \text{ if \autoref{HP noise isotropo} holds}
\end{cases}
\end{align*}
and by $\Lambda_{\rho}^{n}$ the differential operator 
\begin{align*}
\Lambda_{\rho}^{n}F&=-\sum_{l\in \{1,2\}} \partial_l \overline{Q^{n}_{\rho}}(0)\cdot\nabla F_l=-\sum_{\substack{j\in \{1,2,3\}\\l\in \{1,2\}}} \partial_l \overline{Q^{n,\cdot, j}_{\rho}}(0)\partial_jF_l\\
&=\nabla\times(\mathcal{A}^n_{\rho}F),
\end{align*}
where
\begin{align*}
    \mathcal{A}^n_{\rho}=
   \frac{\rho\zeta^n_{H,2}} {C_{1,H}C_{2,H}}\begin{bmatrix} 
   1 & 0 & 0 \\ 0 & 1 & 0\\ 0 & 0 & 0
   \end{bmatrix}.
\end{align*}
In particular it holds
\begin{align*}
(\Lambda_{\rho}^{n}F)_H=\frac{\rho \zeta^n_{H,2}} {C_{1,H}C_{2,H}}\partial_3F_H^{\perp},\quad (\Lambda_{\rho}^{n}F)_3=-\frac{\rho \zeta^n_{H,2}} {C_{1,H}C_{2,H}}\operatorname{div}_H(F_H^{\perp}).    
\end{align*}
\end{lemma}    
According to \autoref{lemma ito strat corrector}, equation \eqref{introductory equation stratonovich} can be rewritten in It\^o form as:
\begin{equation}\label{Introductory equation Ito}
    \begin{cases}
        dB^{n}_{t}&=\left(\eta \Delta+\Lambda^{n}+\Lambda^{n}_{\rho}\right)B^{n}_{t}  dt + \sum_{\substack{k\in \Z^{3}_0\\ j\in \{1,2\}}}\mathcal{L}_{\skj}B^{n}_t   dW^{k,j}_t, \\
        \operatorname{div}B^{n}_t&=0, \\
        B^{n}_t|_{t=0}&=B^{n}_0,
    \end{cases}
\end{equation}
We give now the definition of weak solution for the Stochastic equation of the magnetic field  \eqref{Introductory equation Ito}.

\begin{definition}\label{well posed def}
A stochastic process \begin{align*}
    B^{n}\in C_{\mathcal{F}}([0,T];\mathbf{L}^2)\cap L^2_{\mathcal{F}}([0,T];\mathbf{H}^1)
\end{align*} is a weak solution of equation \eqref{Introductory equation Ito} if, for every $\phi\in \mathbf{H}^2$, we have
\begin{align}\label{weak formulation full system}
    \langle B^{n}_t,\phi\rangle
    &=  \langle B^{n}_0,\phi\rangle+ \eta \int_0^t  \langle B^{n}_s,\Delta\phi\rangle \, ds+\int_0^t  \langle B^{n}_s,\Lambda^{n}\phi\rangle\, ds+\int_0^t  \langle \Lambda^{n}_{\rho}B^{n}_s,\phi\rangle\, ds+\sum_{\substack{k\in \Z^{3}_0\\ j\in \{1,2\}}}\int_0^t\langle \mathcal{L}_{\skj}B^{n}_s,\phi\rangle   dW^{k,j}_s.
\end{align}
for every $t\in [0,T],\ \mathbb{P}-a.s.$
\end{definition}
\begin{remark}\label{general_duality_pairing}
Due to the definition of the operator $P$, each vector field $\phi\in \Dot{H}^{2}(\T^3;\R^3)$ can be decomposed in a divergence free part, i.e. $P\phi\in \mathbf{H}^2$, and a part orthogonal to $\mathbf{L}^2$, i.e. $(I-P)\phi\in \Dot{H}^{2}(\T^3;\R^3)$. Therefore for a generic $\phi\in \Dot{H}^{2}(\T^3;\R^3) $ we get
\begin{align*}
 \langle B^{n}_t,\phi\rangle-\langle B^{n}_0,\phi\rangle&= \langle B^{n}_t,P\phi\rangle-\langle B^{n}_0,P\phi\rangle\\
    &=  \eta \int_0^t  \langle B^{n}_s,\Delta P\phi\rangle \, ds+\int_0^t  \langle B^{n}_s,\Lambda^{n}P\phi\rangle\, ds+\int_0^t  \langle \Lambda^{n}_{\rho}B^{n}_s,P\phi\rangle\, ds\\ &+\sum_{\substack{k\in \Z^{3}_0\\ j\in \{1,2\}}}\int_0^t\langle \mathcal{L}_{\skj}B^{n}_s,P\phi\rangle   dW^{k,j}_s\\ & =  \eta \int_0^t  \langle B^{n}_s,\Delta \phi\rangle \, ds+\int_0^t  \langle B^{n}_s,\Lambda^{n}\phi\rangle\, ds+\int_0^t  \langle \Lambda^{n}_{\rho}B^{n}_s,\phi\rangle\, ds+\sum_{\substack{k\in \Z^{3}_0\\ j\in \{1,2\}}}\int_0^t\langle \mathcal{L}_{\skj}B^{n}_s,\phi\rangle   dW^{k,j}_s
\end{align*}
for every $t\in [0,T],\ \mathbb{P}-a.s.$, due to the fact that $B^n\in C_{\mathcal{F}}([0,T];\mathbf{L}^2)\cap L^2_{\mathcal{F}}([0,T];\mathbf{H}^1) $ and the structure of the operators $\Lambda^n,\ \Lambda^n_{\rho}$. In particular, if we choose $\phi=(0,0,\varphi)\in \Dot{H}^2(\T^3;\R^3)$ we obtain
\begin{align*}
    \langle B^{n,3}_t,\varphi\rangle-\langle B^{n,3}_0,\varphi\rangle&=\eta \int_0^t  \langle B^{n,3}_s,\Delta \varphi\rangle \, ds+\int_0^t  \langle B^{n,3}_s,\Lambda^{n}\varphi\rangle\, ds+\frac{\rho \zeta^n_{H,2}} {C_{1,H}C_{2,H}}\int_0^t  \langle {B^{n,H}_s}^{\perp},\nabla_H\varphi\rangle\, ds\\ &+\sum_{\substack{k\in \Z^{3}_0\\ j\in \{1,2\}}}\int_0^t\langle \skj\cdot\nabla B^{n,3}_s-B^{n,3}\cdot\nabla\skjcomp{3},\varphi\rangle   dW^{k,j}_s
\end{align*}
for every $t\in [0,T],\ \mathbb{P}-a.s.$
\end{remark}
Due to the fact that equation \eqref{Introductory equation Ito} is linear, the existence and uniqueness of solutions in the sense of \autoref{well posed def} is a standard fact which follows by the abstract theory of \cite[Chapters 3-5]{Flandoli_Book_95}, see also \cite[Section 3.1]{flandoli2022heat}. Indeed the following holds.

\begin{theorem}\label{thm well posed}
    For each $B^{n}_0\in \mathbf{L}^2$ there exists a unique weak solution of system \eqref{Introductory equation Ito} in the sense of \autoref{well posed def}.
\end{theorem}
As discussed in \autoref{sec introduction}, we study the convergence of $\Bt{t}$ in the scaling limit of the separation of scales, i.e. considering a noise which concentrates on smaller and smaller scales as described in \autoref{subsect description noise}.
Lastly we introduce our limit objects. Following the idea first introduced in \cite{galeati2020convergence}, we expect that our limit objects satisfy a PDE with an additional second order operator with intensity related to \begin{align*}
    \lim_{n \rightarrow +\infty}\sum_{\substack{k\in \Z^{3}_0\\ j\in \{1,2\}}}\norm{\skj}^2.
\end{align*}
The quantity above is different from $0$ if $\alpha=2$ in \autoref{HP noise 1}, \autoref{HP noise 2} or if it holds \autoref{HP noise isotropo}.
Besides to the fact that the limit above may be $0$ under our assumptions, the final deterministic objects keep memory of our noise thanks to the fact that $\alpha+\beta=6$ for each $n$, which implies that our noise has a nontrivial limit helicity. Denoting by $\overline{B_t}$ the unique weak solutions of the following linear 3D PDE 
\begin{align}\label{limit solution A3}
&\begin{cases}
\partial_t \overline{B_t}&=(\eta\Delta+ \Lambda_{\alpha,\beta,\gamma})\overline{B_t}+ \nabla\times(\mathcal{A}_{\rho}\overline{B_t})\quad x\in \T^3,\ t\in (0,T)\\
\operatorname{div}\overline{B}&=0\quad x\in \T^3,\ t\in (0,T)\\
\overline{B_t}|_{t=0}&=\overline{B_0},   
\end{cases}
\end{align}
where 
\begin{align*}
\Lambda_{\alpha,\beta,\gamma}&=\begin{cases}
    0 \quad & \text{if \autoref{HP noise 1} holds and $\alpha> 2$,}\\
    \eta_T\Delta_H\quad & \text{if \autoref{HP noise 1} holds and $\alpha= 2$,}\\
    \eta_{iso}\Delta \quad & \text{if \autoref{HP noise isotropo} holds}, 
\end{cases}\quad
\mathcal{A}_{\rho}=
   \frac{2\pi\rho\log{2}} {C_{1,H}C_{2,H}}\begin{bmatrix} 
   1 & 0 & 0 \\ 0 & 1 & 0\\ 0 & 0 & 0
   \end{bmatrix},
\end{align*}
 $\eta_T=\frac{\zeta_{H,2}}{2C_{1,H}^2},\ \eta_{iso}=\frac{2\zeta_{s,3}}{3 C_V^2}$.
Our main result reads as follows:
\begin{theorem}\label{main Theorem}
Let $B_0^n\in \mathbf{L}^2$, assuming \autoref{HP noise 1} or \autoref{HP noise isotropo} and $
B^n_0\rightharpoonup \overline{B_0} \text{ in }  \mathbf{H}^{-1}, $ for each  $\kappa\in [1,2),\ \theta\in (0,2),\ \delta\in (0,\theta)$ we have
\begin{align}\label{main thm ineq 1}
\sup_{t\in [0,T]}\expt{\norm{B^n_t- \overline{B_t}}_{\mathbf{H}^{-1-\theta}}^\kappa}  &\lesssim \begin{cases}
    \frac{1}{n^{\kappa\theta \chi(\alpha,\beta,\gamma)/2}}+\frac{1}{n^{\frac{\kappa(\alpha\wedge\beta\wedge\gamma)(\theta-\delta)}{6+\delta}}}+\norm{\overline{B_0}-B^n_0}^\kappa_{\mathbf{H}^{-1-\theta}}&   \text{in case of \autoref{HP noise 1} },\\
    \frac{1}{n^{\frac{3\kappa(\theta-\delta)}{6+\delta}}}+\norm{\overline{B_0}-B^n_0}^\kappa_{\mathbf{H}^{-1-\theta}}&  \text{in case of \autoref{HP noise isotropo} }.
    \end{cases}\end{align}
Moreover, assuming also $B^{n,3}\rightharpoonup\overline{B^{3}_0}$ in $\Dot{L}^2(\T^3;\R^3)$, in case of \autoref{HP noise 2} we have also for each $\vartheta\in (\theta,\frac{3}{2}]$
\begin{align}
\label{main thm ineq 2} \sup_{t\in [0,T]}\expt{\norm{B^{n,3}_t-\overline{B^3_t}}_{\Dot H^{-\vartheta}}^\kappa}&\lesssim   \frac{1 }{n^{\frac{\kappa(\vartheta-\theta)(\theta-\delta)}{6+\delta}}} +\frac{1}{(\vartheta-\theta)^{\kappa}n^{\kappa\theta/2}}+\norm{\overline{B^3_0}-B^{n,3}_0}^{\kappa}_{\Dot{H}^{-\vartheta}}+\frac{\norm{\overline{B_0}-B^n_0}_{\mathbf{H}^{-1-\theta}}^{\kappa}}{(\vartheta-\theta)^{\kappa}}.
\end{align}
In particular either the right hand side of \eqref{main thm ineq 1} and of \eqref{main thm ineq 2} converge to $0$ as $n\rightarrow +\infty$.

\end{theorem}
\begin{remark}
    The hidden constant depends on all the parameters of the model and blows-up as $\delta\rightarrow0$ or if either $\eta_{iso}\nearrow \eta$ in case of \autoref{HP noise isotropo} or $2\eta_T\nearrow \eta$ in case of \autoref{HP noise 1} and $\alpha=2.$ 
\end{remark}
\begin{remark}\label{dimension turbulent viscosity}
Due to \autoref{HP noise isotropo} $\eta_{iso}<\eta$. In case of \autoref{HP noise 1} with $\alpha=2$ we have $\eta_T<\frac{\eta}{2}$.   
\end{remark}
Contrary to our previous work \cite{butori2024stratonovich}, here also a compactness approach inspired to \cite[Section 5]{flandoli2023boussinesq} works. Indeed, arguing similarly to \autoref{a_priori_bound_stocastic_conv}, it is possible to study the behaviour of the time increments of $B^n_t$ in some Sobolev spaces of negative order. Combining this control and the uniform estimates of \autoref{prop:compact_space} allows to prove the tightness of the laws of $\{B^n\}_{n\in \N}$ in $L^2(0,T;H^{-\delta})$ for each $\delta>0$. Then, arguing by Prokhorov's theorem and Skorokhod's representation theorem, it is possible to find an auxiliary probability space where, up to passing to subsequences, $B^n\rightarrow \overline{B}$ in $L^2(0,T;H^{-\delta})\quad \mathbb{P}-a.s.$ Arguing similarly to \autoref{sec proof main thm}, we can identify $\overline{B}$ as the unique weak solution of \eqref{limit solution A3}. This implies that actually the full sequence $B^n\rightarrow \overline{B}$ in $L^2(0,T;H^{-\delta})$ in probability. Similar arguments applies also to $B^{n,3}$ in case of \autoref{HP noise 2}. We preferred to rely on semigroup techniques and state our main result in the form of \autoref{main Theorem} for a twofold reason. First, this quantitative result gives explicit information on the rate of convergence which we think are of independent interest. Secondly, the convergence in $C([0,T];L^2(\Omega;\mathbf{H}^{-1-\theta}))$ guaranteed by \autoref{main Theorem} gives us more information about the qualitative behaviour of the stochastic model. Indeed, as a corollary of \autoref{main Theorem} and its proof, in particular \autoref{prop:compact_space}, we can easily obtain \autoref{dynamo_effect} and \autoref{beta_effect}.
\begin{corollary}\label{dynamo_effect}
    Let $k=\lambda e_3,$ for $\lambda \in \Z_0$ and $B_0^n=\overline{B_0}=(\sin(k\cdot x), \cos(k\cdot x),0)^t$. Assuming either \autoref{HP noise 1} or \autoref{HP noise 2}, if $\lambda\rho>0$ and $ C_{2,H}<
        \frac{2\pi\rho\log 2}{\lambda C_{1,H}{{\eta}}},$
for $n$ large enough, $\expt{\norm{B^n_t}^2}$ grows exponentially in time for each $t\geq 0$. In particular
\begin{align}\label{mean_dynamo}
    \expt{\norm{B^n_t}^2}\geq e^{-(\lambda (\mathcal{H}+O(n^{-1}))  +2(\eta+O(n^{2-\beta}+n^{3-\gamma})\lambda^2)t}\norm{\overline{B}_0}^2\quad \forall t\geq 0.
\end{align}  
Moreover, for each $\eps,\eps', T> 0$, there exists $\overline{n}\in \N$ large enough such that for each $n\geq \overline{n}$ there exists $\mathcal{N}_n\subseteq \Omega$ which satisfies $\mathbb{P}(\mathcal{N}_n)\leq \eps'$ and \begin{align}\label{pathwise_dynamo}
    \norm{B^n_t}\geq e^{-(\lambda  \frac{\mathcal{H}^n}{2}+(\eta+\eta^{n}_R+\eta^n_V-\eta^n_{R,V} )\lambda^2)t}(1-\epsilon)\|\overline{B}_0\| \quad \forall t\in [0,T],\ \forall\omega\in \mathcal{N}_n^c.
\end{align}
\end{corollary}
\begin{proof}
    Let us denote by $b^n_{\lambda,t}:=\langle B_t^n,\overline{B}_0\rangle$. Due to the weak formulation satisfied by $B^n_t$ we obtain an equation for the evolution of $b^n_{\lambda,t}$:
\begin{align}\label{SDE}
    db^n_{\lambda,t} = \sumkj\langle B_t^n, ( \nabla\times \overline{B}_0)\times \skj \rangle dW_t^k + \langle \Lambda^n_\rho B_t^n, \overline{B}_0\rangle dt - (\eta+\eta^{n}_R+\eta^n_V-\eta^n_{R,V} )\lambda^2 b_{\lambda,t}^ndt
\end{align}
Now, integrating by parts we can rewrite the term with $\Lambda_\rho^n$ as 
\begin{align*}
    \langle \Lambda^n_\rho B_t^n, \overline{B}_0\rangle = -\lambda  \frac{\mathcal{H}^n} {2} b_{\lambda, t}^n
\end{align*}
Taking the expected value we immediately get 
\begin{align*}
    \expt{b_{\lambda,t}^n}= e^{-(\lambda  \frac{\mathcal{H}^n}{2}+(\eta+\eta^{n}_R+\eta^n_V-\eta^n_{R,V} )\lambda^2)t}\|\overline{B}_0\|^2.
\end{align*}
Under our assumptions on $\lambda,\ \rho, \ C_{1,H},\ C_{1,H}, \ \eta$, if $n$ is sufficiently large then
\begin{align*}
   \lambda  \frac{\mathcal{H}^n}{2}+(\eta+\eta^{n}_R+\eta^n_V-\eta^n_{R,V} )\lambda^2=\lambda (\mathcal{H}+O(n^{-1}))  +2(\eta+O(n^{2-\beta}+n^{3-\gamma})\lambda^2<0
\end{align*}
and we see exponential growth of $\expt{b_{\lambda,t}^n}$. In particular, by Jensen's inequality, the latter implies 
\begin{align*}
   \expt{\|B_t^n\|^2}\geq \frac{\expt{\lvert b^n_{\lambda,t}\rvert^2}}{\norm{\overline{B}_0}^2}\geq e^{-(\lambda (\mathcal{H}+O(n^{-1}))  +2(\eta+O(n^{2-\beta}+n^{3-\gamma}))\lambda^2)t}\norm{\overline{B}_0}^2.
\end{align*}
Concerning the second statement, due to \eqref{SDE}, $b^{n}_{\lambda,t}$ can be written as 
\begin{align*}
    b^{n}_{\lambda,t}&=e^{-(\lambda  \frac{\mathcal{H}^n}{2}+(\eta+\eta^{n}_R+\eta^n_V-\eta^n_{R,V} )\lambda^2)t}\norm{\overline{B}_0}^2\\ & +\sumkj\int_0^t e^{-(\lambda  \frac{\mathcal{H}^n}{2}+(\eta+\eta^{n}_R+\eta^n_V-\eta^n_{R,V} )\lambda^2)(t-s)}\langle B_s^n, ( \nabla\times \overline{B}_0)\times \skj \rangle dW_s^k.
\end{align*}
In particular
\begin{multline*}
    \inf_{t\in [0,T]}e^{(\lambda  \frac{\mathcal{H}^n}{2}+(\eta+\eta^{n}_R+\eta^n_V-\eta^n_{R,V} )\lambda^2)t}b^{n}_{\lambda,t}\\ \geq \norm{\overline{B}_0}^2-\sup_{t\in [0,T]}\lvert \sumkj\int_0^t e^{(\lambda  \frac{\mathcal{H}^n}{2}+(\eta+\eta^{n}_R+\eta^n_V-\eta^n_{R,V} )\lambda^2)s}\langle B_s^n, ( \nabla\times \overline{B}_0)\times \skj \rangle dW_s^k\rvert
\end{multline*}
and by Markov and Burkholder-Davis-Gundy inequality
\begin{multline*}
    \mathbb{P}(\inf_{t\in [0,T]}e^{(\lambda  \frac{\mathcal{H}^n}{2}+(\eta+\eta^{n}_R+\eta^n_V-\eta^n_{R,V} )\lambda^2)t}b^{n}_{\lambda,t}\geq (1-\epsilon)\norm{\overline{B}_0}^2)\\ \geq  
    \mathbb{P}(\sup_{t\in [0,T]}\lvert \sumkj\int_0^t e^{(\lambda  \frac{\mathcal{H}^n}{2}+(\eta+\eta^{n}_R+\eta^n_V-\eta^n_{R,V} )\lambda^2)s}\langle B_s^n, ( \nabla\times \overline{B}_0)\times \skj \rangle dW_s^k\rvert\leq   \epsilon\norm{\overline{B}_0}^2)\\  =1-\mathbb{P}(\sup_{t\in [0,T]}\lvert \sumkj\int_0^t e^{(\lambda  \frac{\mathcal{H}^n}{2}+(\eta+\eta^{n}_R+\eta^n_V-\eta^n_{R,V} )\lambda^2)s}\langle B_s^n, ( \nabla\times \overline{B}_0)\times \skj \rangle dW_s^k\rvert>   \epsilon\norm{\overline{B}_0}^2)\\ \geq 1-\frac{2}{\epsilon^2 \norm{\overline{B}_0}^4}\sumkj\expt{\int_0^T e^{2(\lambda  \frac{\mathcal{H}^n}{2}+(\eta+\eta^{n}_R+\eta^n_V-\eta^n_{R,V} )\lambda^2)s}\langle B_s^n, ( \nabla\times \overline{B}_0)\times \skj \rangle^2 ds}.
\end{multline*}
We are left on estimating the last integral, but due to \autoref{prop:compact_space} and our choice of $\skj$ we have
\begin{align*}
    \sumkj\expt{\int_0^T e^{2(\lambda  \frac{\mathcal{H}^n}{2}+(\eta+\eta^{n}_R+\eta^n_V-\eta^n_{R,V} )\lambda^2)s}\langle B_s^n, ( \nabla\times \overline{B}_0)\times \skj \rangle^2 ds}&\leq \frac{\lambda^2\norm{\overline{B}_0}_{L^{\infty}}^2}{n^{\alpha\wedge\beta\wedge\gamma}}\expt{\int_0^T \norm{B^n_t}^2 dt}\\& \lesssim \frac{\lambda^2\norm{\overline{B}_0}^2\norm{\overline{B}_0}_{L^{\infty}}^2}{n^{\alpha\wedge\beta\wedge\gamma}}.
\end{align*}
As a consequence, choosing $n$ large enough
\begin{align*}
    \mathbb{P}(\inf_{t\in [0,T]}e^{(\lambda  \frac{\mathcal{H}^n}{2}+(\eta+\eta^{n}_R+\eta^n_V-\eta^n_{R,V} )\lambda^2)t}b^{n}_{\lambda,t}\geq (1-\epsilon)\norm{\overline{B}_0}^2)\geq 1-\eps'
\end{align*}
and the result follows.
\end{proof}
\begin{corollary}\label{beta_effect}
If \autoref{HP noise isotropo} holds and $B_0^n=\overline{B_0}\in \mathbf{L^2}$, then for each $\kappa\in [1,2),\ \theta\in (0,2),\ \delta\in (0,\theta)$ we have
\begin{align*}
    \expt{\norm{B^n_t}_{\mathbf{H^{-1-\theta}}}^\kappa}\leq \frac{C}{n^{\frac{3\kappa(\theta-\delta)}{6+\delta}}}+2^{\kappa-1}\norm{\overline{B_0}}^\kappa e^{-\kappa\left(\eta+\frac{(8-6\lvert \rho\rvert)\pi\log 2}{3C_V^2}\right)t}\quad \forall t\in [0,T],
\end{align*}
where $C$ is a constant depending on $\overline{B_0},C_V,\rho,\theta,\delta,\kappa, T$.
In particular, for each $\epsilon,\eps'>0$ there exists $\overline{n}\in \N$ large enough such that for each $t\in [0,T], \ n\geq \overline{n}$ 
\begin{align*}
    \mathbb{P}\left(\norm{B^n_t}_{\mathbf{H^{-1-\theta}}}\geq (1+\epsilon)\norm{\overline{B}_0}e^{-\left(\eta+\frac{(8-6\lvert \rho\rvert)\pi\log 2}{3C_V^2}\right)t}\right)\leq \eps'.
\end{align*}
\end{corollary}
\begin{proof}
    Thanks to simple energy estimates on \eqref{limit solution A3} and Poincaré inequality we obtain
    \begin{align*}
\norm{\overline{B_t}}^2+2\left(\eta+\frac{8\pi\log 2}{3C_V^2}\right)\int_0^t \norm{\nabla \overline{B_s}}^2 ds&\leq \norm{\overline{B_0}}^2+\frac{4\pi\lvert \rho\rvert \log 2}{C_V^2}\int_0^t \norm{\nabla \overline{B_s}}^2 ds.
    \end{align*}
Therefore, by Gr\"onwall's inequality
\begin{align}\label{a_priori_estimate_enhanced}
\norm{\overline{B_t}}^\kappa\leq \norm{\overline{B_0}}^\kappa e^{-\kappa\left(\eta+\frac{(8-6\lvert \rho\rvert)\pi\log 2}{3C_V^2}\right)t}.    
\end{align}
As a consequence of the inequality above and \autoref{main Theorem} we obtain
\begin{align*}
    \expt{\norm{B^n_t}_{\mathbf{H}^{-1-\theta}}^\kappa}&\leq 2^{\kappa-1}\left(\expt{\norm{B^n_t-\overline{B_t}}_{\mathbf{H}^{-1-\theta}}^\kappa}+\norm{\overline{B_t}}_{\mathbf{H}^{-1-\theta}}^\kappa\right)\\ & \leq \frac{C_{T,B_0,C_V,\rho,\theta,\delta,\kappa }}{n^{\frac{3\kappa(\theta-\delta)}{6+\delta}}}+2^{\kappa-1}\norm{\overline{B_0}}^\kappa e^{-\kappa\left(\eta+\frac{(8-6\lvert \rho\rvert)\pi\log 2}{3C_V^2}\right)t}.
\end{align*}
Concerning the second result, since by triangle inequality, Sobolev embedding and \eqref{a_priori_estimate_enhanced}
\begin{multline*}
    \mathbb{P}(\norm{B^n_t}_{\mathbf{H}^{-1-\theta}}e^{\left(\eta+\frac{(8-6\lvert \rho\rvert)\pi\log 2}{3C_V^2}\right)t}\geq (1+\epsilon)\norm{\overline{B}_0})\\ \leq \mathbb{P}(\norm{\overline{B}_t}e^{\left(\eta+\frac{(8-6\lvert \rho\rvert)\pi\log 2}{3C_V^2}\right)t}+\norm{B^n_t-\overline{B}_t}_{\mathbf{H}^{-1-\theta}}e^{\left(\eta+\frac{(8-6\lvert \rho\rvert)\pi\log 2}{3C_V^2}\right)t}\geq (1+\epsilon)\norm{\overline{B}_0})\\  \leq 
    \mathbb{P}(\norm{B^n_t-\overline{B}_t}_{\mathbf{H}^{-1-\theta}}e^{\left(\eta+\frac{(8-6\lvert \rho\rvert)\pi\log 2}{3C_V^2}\right)t}\geq \epsilon\norm{\overline{B}_0})\leq \frac{\expt{\norm{B^n_t-\overline{B}_t}_{\mathbf{H}^{-1-\theta}}^{\kappa}}}{\eps^\kappa e^{-\kappa\left(\eta+\frac{(8-6\lvert \rho\rvert)\pi\log 2}{3C_V^2}\right)T}\norm{\overline{B}_0}^\kappa},
\end{multline*}
where in the last step we applied Markov inequality. The claim then follows by \autoref{main Theorem}.
\end{proof}
The two corollaries above have a clear interpretation in terms of the discussion of \autoref{sec introduction} and \autoref{subsect description noise}. Indeed, \autoref{dynamo_effect} shows that if the helicity of our velocity field is large enough and the turbulent fluid is sufficiently anisotropic, the energy of our stochastic models grows exponentially in time. Therefore, a dynamo effect is active on our stochastic models in case of \autoref{HP noise 1}. On the contrary, in case of \autoref{HP noise isotropo}, due to \autoref{beta_effect}, the dissipation acting on magnetic field increases due to the presence of the turbulent fluid. This is accordance to \cite{krause2016mean}. Moreover, \autoref{beta_effect} implies that it is not needed that the fluid is either completely isotropic or mirror symmetric in order to increase the dissipative effects acting on the magnetic field.  Let us stress more the links between \autoref{dynamo_effect} and the theory of dynamo for the passive magnetic field. Relation \eqref{mean_dynamo} implies
    \begin{align*}
\limsup_{t\rightarrow+\infty}\frac{\log\expt{\norm{B^n_t}^2}}{t}>0.
    \end{align*}
    The quantity above corresponds with the definition of kinematic dynamo for stochastic models discussed in \cite[Chapter 11]{childress2008stretch}. However, even if this definition of kinematic dynamo is the most tractable mathematically for stochastic problems, it is not the only definition, and perhaps not the best definition. Indeed, it does not provide any information about the typical realizations of our process as discussed in \cite[Section 11.5]{childress2008stretch}, this is the content of our second pathwise statement \eqref{pathwise_dynamo}. The latter even if restricted to finite, arbitrary long, time intervals says exactly that typical trajectories of the magnetic field  (and not only their mean) follow exponential growth. The restriction of \eqref{pathwise_dynamo} to finite time intervals is a minor drawback of our result. Indeed, the equation for the passive magnetic field is only an approximation of the nonlinear equations of Magnetohydrodynamics becoming not valid anymore as the magnetic field becomes to huge and its effect of the velocity of the fluid cannot be neglected, see for example \cite[Chapter 1-7]{zheligovsky2011large}, \cite{lanotte1999large}. As discussed in \cite[Section 1.2.4]{childress2008stretch}, \emph{the existence of a kinematic dynamo ensures that a non-magnetic state is unstable. The instability is excited by introducing a weak seed magnetic field, far too feeble to affect the flow. The growth of the seed field will then ensue, until the magnetic stresses build up to a point where the flow is affected and a full Magnetohydrodynamics theory is required to determine the subsequent evolution.} Relation \eqref{pathwise_dynamo}, exactly says that with arbitrarily large probability, we reach, tuning the parameters properly also in arbitrarily small time, a turning point where full Magnetohydrodynamics theory is required to determine the subsequent evolution.
\begin{remark}
    Contrary to \autoref{HP noise isotropo}, the matrix $Q^n(0)$ converge to a degenerate matrix as $n\rightarrow +\infty$ in case of \autoref{HP noise 1}. This phenomena produces a degenerate second order operator in the scaling limit. It is possible to avoid such a phenomena in case of \autoref{HP noise 1} and $\alpha=2$ considering a more general noise
    \begin{align*}
        W^n_t=\frac{1}{3}\sumkj a_{k,j}e^{ik\cdot x}\left(\theta^n_{1,k,j} W^{1,k,j}_t+\theta^n_{2,k,j} W^{2,k,j}_t+\theta^n_{3,k,j} W^{3,k,j}_t\right),
    \end{align*}
    where $W^{h,k,j}_t$ are complex Brownian motions satisfying 
    \begin{align*}
    \expt{W^{h,k,j}_1,{W^{h',l,m}_1}^*}&=\begin{cases}0 &\textit{if } h\neq h'\\
        2 & \textit{if } h=h',\ k=l, \ m=j\\
        2\rho_h &  \textit{if } h=h',\  k=l,\ k_h=0,\ m\neq j
    \end{cases}
    \end{align*}
    and \begin{align*}
   \theta_{h,k,j}^n &=\one_{\{n\leq \lvert k\rvert\leq 2n\} }\begin{cases}
        \frac{i}{C_{1,H}\lvert k\rvert^{\alpha/2}}& \text{if } k_h=0, k\in\Gamma_{+},\ j=1\\
         \frac{-i}{C_{1,H}\lvert k\rvert^{\alpha/2}}& \text{if } k_h=0, k\in\Gamma_{-},\ j=1\\
        \frac{1}{C_{2,H}\lvert k\rvert^{\beta/2}}& \text{if } k_h=0,\ j=2\\
        \frac{1}{C_{V}\lvert k\rvert^{\gamma/2}} & \text{if } k_h\neq 0.
    \end{cases}
\end{align*}
In this case we get 
\begin{align*}
    \Lambda_{\alpha,\beta,\gamma}=\frac{2}{3}\eta_T \Delta,\quad \mathcal{A}_{\rho_1,\rho_2,\rho_3}=\frac{2\pi\log 2}{3C_{1,H}C_{2,H}}\left(\rho_1\begin{bmatrix} 
   0 & 0 & 0 \\ 0 & 1 & 0\\ 0 & 0 & 1
   \end{bmatrix}+\rho_2\begin{bmatrix} 
   1 & 0 & 0 \\ 0 & 0 & 0\\ 0 & 0 & 1
   \end{bmatrix}+\rho_3\begin{bmatrix} 
   1 & 0 & 0 \\ 0 & 1 & 0\\ 0 & 0 & 0
   \end{bmatrix}\right)
\end{align*}
and we can prove an analogous result to \autoref{main Theorem}. Since this change would drastically complicate the notation without any change either in the proof of \autoref{main Theorem} and in the physical understanding of the kinematic dynamo and the dissipative properties of our stochastic systems, i.e. \autoref{dynamo_effect} and \autoref{beta_effect}, we prefer to limit ourselves to the noise introduced in \autoref{subsect description noise}. In this spirit the \emph{helical} regime could be regarded as a stochastic perturbation of a Robert's flow \cite{roberts1972dynamo} in which we can also allow dependence on the third space variable. This kind of flows are extensively studied in the physical and numerical literature as paradigmatic models for dynamo action (see for instance \cite{Brandenburg2003}, \cite{Brandenburg2008}, \cite{courvoisier2005dynamo} and references therein) 
\end{remark}

\section{A Priori Estimates}\label{section a priori estimates}
The goal of this section is to provide uniform bounds for the $\mathbf{H}^{-1}$ norm of our vector fields $B^n_t$. In order to complete our plan it is useful to recall a well-known result about equivalence of norms for divergence free vector fields, see for example \cite{Marchioro1994}.
\begin{proposition}\label{equivalence norms}
Let $X\in \mathbf{L}^2,\ Y\in \mathbf{H}^1$. Then 
\begin{align}\label{property 1 norm}
    \norm{X}^2&=\norm{(-\Delta)^{-1/2}\nabla\times X}^2,\quad \norm{\nabla Y}^2=\norm{\nabla\times Y}^2.
\end{align}

\end{proposition}

We start providing a control to some quantities related to the noise appearing in equation \eqref{Introductory equation Ito}.
\begin{lemma}\label{Ito_Strat_correct_H-1}
Assuming \autoref{HP noise 1}, then
\begin{align*}
\sumkj \lVert \skj \times \Bt{t}\rVert^2\leq  \left(\frac{\zeta^n_{H,\alpha}}{C_{1,H}^2 n^{2-\alpha}}+\frac{\zeta^n_{H,\beta}}{C_{2,H}^2 n^{2-\beta}}+\frac{2\zeta^n_{s,\gamma}}{C_{V}^2 n^{3-\gamma}}\right)\lVert \Bt{t}\rVert^2.  
\end{align*}
In case of \autoref{HP noise isotropo}, then 
\begin{align*}
 \sumkj \lVert \skj \times \Bt{t}\rVert^2=   \frac{4\zeta^n_{s,3}}{3C_V^2}\lVert \Bt{t}\rVert^2.
\end{align*}
In particular, assuming \autoref{HP noise 2} we have also
\begin{align*}
\sumkj \lVert \skj\cdot\nabla B^{n,3}_t\rVert^2&\leq \frac{\zeta_{H,\alpha}^n\lVert \nabla_H B^{n,3}_t\rVert^2}{2C_{1,H}^2n^{\alpha-2}}+\frac{\zeta^n_{H,\beta}\lVert \partial_3 B^{n,3}_t\rVert^2}{C_{2,H}^2n^{\beta-2}}+\frac{2\zeta^n_{s,\gamma}\lVert \nabla B^{n,3}_t\rVert^2}{3C_V^2n^{\gamma-3}},\\
\sumkj \lVert B^n_t\cdot\nabla \skjcomp{3}\rVert^2&\leq \left(\frac{\zeta^n_{H,2}}{C_{2,H}^2}+\frac{\zeta^n_{s,\gamma-2}}{C_{V}^2n^{\gamma-5}}\right)\lVert B^n_t\rVert^2.
\end{align*}
\end{lemma}
\begin{proof}
    By H\"older's inequality we have
    \begin{align*}
    \sumkj \lVert \skj \times \Bt{t}\rVert^2\leq \sumkj \lvert \tkj\rvert^2\lVert \Bt{t}\rVert^2.  
    \end{align*}
    In the first case, due to the relation above we get
    \begin{align*}
    \sumkj \lVert \skj \times \Bt{t}\rVert^2& \leq \left(\sumkmeanjuno \lvert \tkj\rvert^2+\sumkmeanjdue \lvert \tkj\rvert^2+\sumkoscj \lvert \tkj\rvert^2 \right) \lVert \Bt{t}\rVert^2\\&\leq \left(\sum_{\substack{k\in \Z^2_0\\
    n\leq \lvert k\rvert \leq 2n}}\frac{1}{C_{1,H}^2\lvert k\rvert^{\alpha}}+\sum_{\substack{k\in \Z^2_0\\
    n\leq \lvert k\rvert \leq 2n}} \frac{1}{C_{2,H}^2\lvert k\rvert^{\beta}}+2\sum_{\substack{k\in \Z_0^3\\ n\leq \lvert k\rvert \leq 2n}} \frac{1}{C_{V}^2\lvert k\rvert^{\gamma}} \right)\lVert \Bt{t}\rVert^2\\ & \leq \left(\frac{\zeta^n_{H,\alpha}}{C_{1,H}^2 n^{2-\alpha}}+\frac{\zeta^n_{H,\beta}}{C_{2,H}^2 n^{2-\beta}}+\frac{2\zeta^n_{s,\gamma}}{C_{V}^2 n^{3-\gamma}}\right)\lVert \Bt{t}\rVert^2. 
    \end{align*}
    However, in the second case, due to the symmetries of the coefficients, we can obtain a sharp relation which is our claim. Let us denote by 
    \begin{align*}
        b^{n,h,l}_t:=\langle \Bt{t}, a_{h,l}e^{-ih\cdot x}\rangle,
    \end{align*}
    therefore
    by definition
    \begin{align*}
    \Bt{t}&=\sum_{\substack{h\in \Z^3_0,\\ l\in \{1,2\}}}    b^{n,h,l}_t a_{h,l}e^{ih\cdot x},\\
    \sigma^n_{k,j}\times \Bt{t}&=\tkj\sum_{\substack{h\in \Z^3_0,\\ l\in \{1,2\}}}   b^{n,h,l}_t a_{k,j}\times a_{h,l}e^{i(h+k)\cdot x}
    \end{align*}
    and for each $k,j$ it holds
    \begin{align}\label{estimate_precise_cancellation}
    \lVert \skj \times \Bt{t}\rVert^2&=\lvert\tkj\rvert^2\sum_{\substack{h\in \Z^3_0,\\ l,l'\in \{1,2\}}}\left(a_{k,j}\times a_{h,l} \right)\cdot\left(a_{k,j}\times a_{h,l'} \right)b^{n,h,l}_t\overline{b^{n,h,l'}_t}\notag\\ & =\lvert\tkj\rvert^2\sum_{\substack{h\in \Z^3_0,\\ l,l'\in \{1,2\}}}\left(\delta_{l,l'}-a_{k,j}\otimes a_{k,j}:a_{h,l}\otimes a_{h,l'}\right)b^{n,h,l}_t\overline{b^{n,h,l'}_t}\notag\\ & =\lvert\tkj\rvert^2 \lVert \Bt{t}\rVert^2- \lvert\tkj\rvert^2\sum_{\substack{h\in \Z^3_0,\\ l,l'\in \{1,2\}}}a_{k,j}\otimes a_{k,j}:a_{h,l}\otimes a_{h,l'}b^{n,h,l}_t\overline{b^{n,h,l'}_t}.
    \end{align}
    Therefore, in case of \autoref{HP noise isotropo} we obtain 
    \begin{align*}
        \sumkj \lVert \skj \times \Bt{t}\rVert^2 &=\frac{2\zeta^n_{s,3}}{C_V^2}\lVert \Bt{t}\rVert^2-\sum_{\substack{h\in \Z^3_0,\\ l,l'\in \{1,2\}}}\frac{2\zeta^n_{s,3}}{3C_V^2}I:\left(a_{h,l}\otimes a_{h,l'}\right)b^{n,h,l}_t\overline{b^{n,h,l'}_t}\\ & =\frac{4\zeta^n_{s,3}}{3C_V^2}\lVert \Bt{t}\rVert^2.
    \end{align*}
Lastly we assume \autoref{HP noise 2}, in this case only few terms in the sum must be considered due to our choice of the $\skj$'s.
Then, we obtain
\begin{multline*}
\sumkj \lVert \skj\cdot\nabla B^{n,3}_t\rVert^2\\=\sum_{\substack{k\in \Z^2_0\\ n\leq \lvert k\rvert \leq 2n}}\left(\frac{\lVert \nabla_H B^{n,3}_t\rVert^2}{2C_{1,H}^2\lvert k\rvert^{\alpha}}+\frac{\lVert \partial_3 B^{n,3}_t\rVert^2}{C_{2,H}^2\lvert k\rvert^{\beta}}\right)+\frac{1}{C_V^2}\sum_{\substack{k\in \Z^3_0,\ k_3\neq 0\\ n\leq \lvert k\rvert \leq 2n}}\frac{1}{\lvert k\rvert^{\gamma}}\left\langle \left(I-\frac{k\otimes k}{\lvert k\rvert^2}\right)\nabla B^{n,3}_t,\nabla B^{n,3}_t\right\rangle\\ \leq \frac{\zeta_{H,\alpha}^n\lVert \nabla_H B^{n,3}_t\rVert^2}{2C_{1,H}^2n^{\alpha-2}}+\frac{\zeta^n_{H,\beta}\lVert \partial_3 B^{n,3}_t\rVert^2}{C_{2,H}^2n^{\beta-2}}+\frac{1}{C_V^2}\sum_{\substack{k\in \Z^3_0\\ n\leq \lvert k\rvert \leq 2n}}\frac{1}{\lvert k\rvert^{\gamma}}\left\langle \left(I-\frac{k\otimes k}{\lvert k\rvert^2}\right)\nabla B^{n,3}_t,\nabla B^{n,3}_t\right\rangle\\ =\frac{\zeta_{H,\alpha}^n\lVert \nabla_H B^{n,3}_t\rVert^2}{2C_{1,H}^2n^{\alpha-2}}+\frac{\zeta^n_{H,\beta}\lVert \partial_3 B^{n,3}_t\rVert^2}{C_{2,H}^2n^{\beta-2}}+\frac{2\zeta^n_{s,\gamma}\lVert \nabla B^{n,3}_t\rVert^2}{3C_V^2n^{\gamma-3}}
\end{multline*}
and by H\"older's inequality
\begin{align*}
\sumkj \lVert B^n_t\cdot\nabla \skjcomp{3}\rVert^2&\leq  \sum_{\substack{k\in \Z^2_0\\ n\leq \lvert k\rvert \leq 2n}}   \frac{\lVert B^n_t\rVert^2}{C_{2,H}^2 \lvert k\rvert^{\beta-2}}+\sum_{\substack{\lvert k\in Z^3_0\\ n\leq \lvert k\rvert\leq 2n}}\frac{\lVert B^n_t\rVert^2}{C_{V}^2 \lvert k\rvert^{\gamma-2}}\\ & \leq \left(\frac{\zeta^n_{H,2}}{C_{2,H}^2}+\frac{\zeta^n_{s,\gamma-2}}{C_{V}^2n^{\gamma-5}}\right)\lVert B^n_t\rVert^2.
\end{align*}
\end{proof}
\begin{remark}
Note that the simple H\"older estimate in case of \autoref{HP noise isotropo} gives us 
    \begin{align*}
    \sumkj \lVert \skj \times \Bt{t}\rVert^2&\leq 2\sum_{\substack{k\in \Z^3_0\\
    n\leq \lvert k\rvert \leq 2n}}\frac{1}{\lvert k\rvert^3 C_V^2}\lVert \Bt{t}\rVert^2\\ & =\frac{2\zeta^n_{s,3}}{C_V^2}\lVert \Bt{t}\rVert^2, 
\end{align*}
which is worse of the second statement of \autoref{Ito_Strat_correct_H-1} due to the fact that we are neglecting the second term of \eqref{estimate_precise_cancellation} which produce some cancellations.
\end{remark}
Now we are ready to prove the main result of this section.
\begin{proposition}\label{prop:compact_space}
Assuming \autoref{HP noise 1} or \autoref{HP noise isotropo} then
\begin{align}\label{compactness space_h-1}
    \expt{\sup_{t\in [0,T]}\norm{B^n_t}_{\mathbf{H}^{-1}}^2}+\expt{\int_0^T \norm{B^n_t}^2} dt\lesssim \norm{B_0}_{\mathbf{H}^{-1}}^2.
\end{align}
In particular if \autoref{HP noise 2} is satisfied then also
\begin{align}\label{compactness space_l2}
 \expt{\sup_{t\in [0,T]}\norm{B^{n,3}_t}^2}+\int_0^T \expt{\norm{\nabla B^{n,3}_t}^2} dt\lesssim \norm{B_0}_{\mathbf{H}^{-1}}^2+\norm{B^{3}_0}^2.    
\end{align}
\end{proposition}
\begin{proof}
Since $\nabla\times(\skj \times B^n_t)=\mathcal{L}_{\skj}B^n_t$, the It\^o formula for $\|B^n_t\|^2_{\mathbf{H}^{-1}}$ reads as
\begin{align}\label{ito_h-1}
    d\|(-\Delta)^{-1/2}B^n_t\|^2 + 2\eta\|B^n_t\|^2dt &= 2\brak{(-\Delta)^{-1}\Lambda^n B^n_t, B^n_t}dt \notag\\
    &+ 2\brak{(-\Delta)^{-1/2}\Lambda_\rho^n B^n_t, (-\Delta)^{-1/2} B^n_t}dt\notag\\
    & + 2\sumkj\|(-\Delta)^{-1/2} \nabla\times (\skj\times B^n_t)\|^2 dt \notag\\ &+ 2\rho\sum_{k\in \Z^2_0} \brak{\nabla\times (\sigma_{k,1}^n\times B^n_t) , (-\Delta)^{-1}\nabla\times(\sigma_{-k, 2}^n\times B^n_t)}dt\notag\\ & +2\rho\sum_{k\in \Z^2_0} \brak{\nabla\times (\sigma_{k,2}^n\times B^n_t) , (-\Delta)^{-1}\nabla\times(\sigma_{-k, 1}^n\times B^n_t)}dt\notag\\ & + dM^n_t,
\end{align}
where 
\begin{align*}
    M^n_t&=2\sumkj\int_0^t\langle(-\Delta)^{-1}\nabla\times(\skj\times B^n_s),B^n_s \rangle dW^{k,j}_s.
\end{align*}
We treat first the case of \autoref{HP noise isotropo} being easier.
By the definition of $\Lambda^n$ in this case we get easily
\begin{align}\label{estimate_1_iso}
\brak{(-\Delta)^{-1}\Lambda^n B^n_t, B^n_t}=-\eta^n_{iso}\norm{B^n_t}^2.   
\end{align}
Concerning the It\^o correctors we have
\begin{align}\label{estimate_2_iso}
\sumkj\|(-\Delta)^{-1/2} \nabla\times (\skj\times B^n_t)\|^2& = \sumkj \langle (-\Delta)^{-1}\nabla\times \nabla\times  (\skj\times B^n_t),\overline{\skj\times B^n_t}\rangle\notag\\ & =   \sumkj \langle (-\Delta)^{-1}\nabla\times \nabla\times P[Q[\skj\times B^n_t]],\overline{\skj\times B^n_t}\rangle\notag\\ & =\sumkj \langle P[Q[\skj\times B^n_t]],\overline{P[Q[\skj\times B^n_t]]}\rangle\notag\\ & \leq \sumkj\lVert \skj\times B^n_t\rVert^2= 2\eta^n_{iso}\lVert \Bt{t}\rVert^2,
\end{align}
due to the properties of the projections $P,\ Q$, relation \eqref{property 1 norm} and \autoref{Ito_Strat_correct_H-1}.
The other terms appearing in \eqref{ito_h-1} can be treated easily by H\"older and Young's inequalities. Indeed, again thanks to relation \eqref{property 1 norm} 
\begin{align}\label{estimate_covariance_part_iso}
    \brak{(-\Delta)^{-1/2}\Lambda_\rho^n B^n_t, (-\Delta)^{-1/2} B^n_t}&\leq \frac{\lvert\rho\rvert\zeta^n_{H,2}} {C_{1,H}C_{2,H}} \lVert  B^n_t\rVert_{\mathbf{H}^{-1}}\norm{(-\Delta)^{-1/2}\nabla\times B^{n,H}_t}\notag\\ & \leq  \frac{\lvert\rho\rvert\zeta^n_{H,2}} {C_{1,H}C_{2,H}} \lVert  B^n_t\rVert_{\mathbf{H}^{-1}}\norm{ B^{n}_t}\notag\\ & \leq \frac{(\eta-\eta_{iso})}{2}\norm{ B^{n}_t}^2+\frac{\lvert \rho\rvert^2\left(\zeta^n_{H,2}\right)^2}{2C_{1,H}^2C_{2,H}^2(\eta-\eta_{iso})} \lVert  B^n_t\rVert_{\mathbf{H}^{-1}}^2
\end{align}
since by assumptions $\eta>\eta_{iso}$ and 
\begin{align}\label{estimate_corrector_covariance_iso}
\rho\sum_{k\in \Z^2_0} \brak{\nabla\times (\sigma_{k,1}^n\times B^n_t) , (-\Delta)^{-1}\nabla\times(\sigma_{-k, 2}^n\times B^n_t)}&\leq \lvert\rho \rvert\sum_{\substack{k\in \Z^2_0\\ n\leq \lvert k\rvert\leq 2n}}\frac{1}{C_{1,H}C_{2,H}\lvert k\rvert^{\alpha/2+\beta/2}}\lVert B^n_t\rVert^2\notag\\ &  \leq  \lvert\rho\rvert \frac{\zeta^n_{H,\alpha/2+\beta/2-2}}{C_{1,H}C_{2,H}n^{\alpha/2+\beta/2-2}}\lVert B^n_t\rVert^2
\end{align}
and analogously for the other. Therefore, integrating \eqref{ito_h-1} between $0$ and $t$ and considering its expected value, thanks to \eqref{estimate_1_iso},\eqref{estimate_2_iso},\eqref{estimate_covariance_part_iso},\eqref{estimate_corrector_covariance_iso} we obtain
\begin{align*}
\expt{\norm{B^n_t}_{\mathbf{H}^{-1}}^2} +\eta   \expt{\int_0^t\norm{B^n_s}^2ds}&\leq \norm{B^n_0}_{\mathbf{H}^{-1}}^2+\left(\eta_{iso}+O(n^{-1})\right)\expt{\int_0^t\norm{B^n_s}^2ds}\\ &+\left(\frac{\lvert \rho\rvert^2\left(\zeta_{H,2}\right)^2}{C_{1,H}^2C_{2,H}^2(\eta-\eta_{iso})} +O(n^{-1})\right)\expt{\int_0^t\lVert  B^n_s\rVert_{\mathbf{H}^{-1}}^2ds}.
\end{align*}
As a consequence, if $n$ is large enough such that $\eta_{iso}+O(n^{-1})<\eta$ we obtain, by Gr\"onwall's lemma
\begin{align}\label{estimate_without_sup_iso}
\sup_{t\in [0,T]}\expt{\norm{B^n_t}_{\mathbf{H}^{-1}}^2} +\expt{\int_0^t\norm{B^n_s}^2ds}\lesssim \norm{B^n_0}_{\mathbf{H}^{-1}}^2.   
\end{align}
Restarting again from \eqref{ito_h-1} and considering the expected value of the supremum in time between $[0,T]$, thanks to \eqref{estimate_1_iso},\eqref{estimate_2_iso},\eqref{estimate_covariance_part_iso},\eqref{estimate_corrector_covariance_iso}, for $n$ large enough such that $\eta_{iso}+O(n^{-1})<\eta$ we obtain 
\begin{multline}\label{estimate_sup time}
    \expt{\sup_{t\in [0,T]}\norm{B^n_t}_{\mathbf{H}^{-1}}^2}\\ \leq \norm{B^n_0}_{\mathbf{H}^{-1}}^2+\left(\frac{\lvert \rho\rvert^2\left(\zeta_{H,2}\right)^2}{C_{1,H}^2C_{2,H}^2(\eta-\eta_{iso})} +O(n^{-1}))\right)T\sup_{t\in [0,T]}\expt{\norm{B^n_t}_{\mathbf{H}^{-1}}^2}+2\expt{\sup_{t\in [0,T]}\left\lvert M^n_t\right\rvert}.
\end{multline}
The last term can be estimated by Burkholder-Davis-Gundy inequality obtaining 
\begin{align}\label{estimate_martingale_term}
\expt{\sup_{t\in [0,T]}\left\lvert M^n_t\right\rvert}& \lesssim_{\lvert \rho\rvert} \expt{\left(\sumkj\int_0^T\langle (-\Delta)^{-1}\nabla\times(\skj\times B^n_s),B^n_s\rangle^2ds\right)^{1/2}}\notag\\ & \leq \expt{\sup_{t\in [0,T]}\norm{ B^n_t}_{\mathbf{H}^{-1}}\left(\sumkj\lvert \tkj\rvert^2\int_0^T  \norm{B^n_s}^2 ds\right)^{1/2}}\notag\\ & \leq \frac{1}{2}\expt{\sup_{t\in [0,T]}\norm{ B^n_t}^2_{\mathbf{H}^{-1}}}+C_{\lvert \rho\rvert,\eta_{iso}}\expt{\int_0^T \norm{B^n_s}^2 ds }.
\end{align}
Combining \eqref{estimate_without_sup_iso},\eqref{estimate_sup time}, \eqref{estimate_martingale_term} we get immediately \eqref{compactness space_h-1} in case of \autoref{HP noise isotropo}.\\
Let us now consider the case of \autoref{HP noise 1}. 
We start by noticing that 
\begin{align*}
    \Lambda^n = \left(\eta_T^n+\eta^n_V-\frac{\eta^n_{R,V}}{2}\right) \Delta + \left(\eta_R^n -\frac{\eta^n_{R,V}}{2}- \eta_T^n\right)\partial_{33}.
\end{align*}
As a consequence we have
\begin{align}\label{estimate_ito_strat_op_non_iso}
    \brak{(-\Delta)^{-1}\Lambda^n B^n_t, B^n_t} &=  -\left(\eta_T^n+\eta^n_V-\frac{\eta^2_{R,V}}{2}\right) \|B_t\|^2 + \pa{\eta_T^n+\frac{\eta^n_{R,V}}{2}- \eta_R^n}\|(-\Delta)^{-1/2}\partial_3B_t\|^2 \notag\\
    &\leq -\eta^n_V \|B^n_t\|^2+\eta^n_R\|B^n_t\|^2+\eta^n_{R,V}\|B^n_t\|^2.
\end{align}
While, arguing as above, due to \autoref{Ito_Strat_correct_H-1}
\begin{align}\label{estimate_corr_non_iso}
    \sumkj\|(-\Delta)^{-1/2} \nabla\times (\skj\times B^n_t)\|^2& \leq  \left(2\eta^n_T+\eta^n_{R}+3\eta^n_V\right)\lVert \Bt{t}\rVert^2.
\end{align}
Estimate \eqref{estimate_covariance_part_iso} and \eqref{estimate_corrector_covariance_iso} continue to hold, possibly replacing $\eta_{iso}$ with $2\eta_T$. In conclusion, combining \eqref{ito_h-1},\eqref{estimate_ito_strat_op_non_iso}, \eqref{estimate_corr_non_iso},\eqref{estimate_covariance_part_iso} and \eqref{estimate_corrector_covariance_iso} we obtain
\begin{align}\label{no_sup_in_time_non_iso}
\expt{\norm{B^n_t}_{\mathbf{H}^{-1}}^2} +\eta   \expt{\int_0^t\norm{B^n_s}^2ds}&\leq \norm{B^n_0}_{\mathbf{H}^{-1}}^2+\left(4\eta^n_T-2\eta_T+O(n^{2-\beta}+n^{3-\gamma})\right)\expt{\int_0^t\norm{B^n_s}^2ds}\notag\\ &+\left(\frac{\lvert \rho\rvert^2\left(\zeta_{H,2}\right)^2}{C_{1,H}^2C_{2,H}^2(\eta-2\eta_{T})} +O(n^{-1})\right)\expt{\int_0^t\lVert  B^n_s\rVert_{\mathbf{H}^{-1}}^2ds}.    
\end{align}
Arguing as above, by Gr\"onwall's lemma and Burkholder-Davis-Gundy inequality we can prove \eqref{compactness space_h-1} starting from \eqref{no_sup_in_time_non_iso}. We omit the easy details. We are left to show the validity of the additional estimate in case of \autoref{HP noise 2}. Due to \autoref{general_duality_pairing} and \eqref{mirror symmetry product 0}, the It\^o formula for $\|B_t^{n,3}\|^2$ reads 
\begin{align}\label{ito_l2_comp3}
    d\|B_t^{n,3}\|^2 &+ 2\left(\eta+\eta^n_V-\frac{\eta^n_{R,V}}{2}\right) \|\grad B_t^{n,3}\|^2dt+2\eta^n_T\|\nabla_H B_t^{n,3}\|^2dt+2\left(\eta^n_R-\frac{\eta^n_{RV}}{2}\right)\|\partial_3 B_t^{n,3}\|^2dt \notag \\ &= \frac{2\rho \zeta^n_{H,2}} {C_{1,H}C_{2,H}}\brak{(B^{n,H}_t)^{\perp}, \nabla_H B_t^{n,3}} dt + 2\sumkj \pa{\|\skj\cdot\nabla B^{n,3}_t\|^2 + \|B^n_t\cdot\nabla \skjcomp{3}\|^2}dt \notag\\ &+ 2\rho \sum_{\substack{k\in \Z^2_0\\ n\leq \lvert k\rvert\leq 2n}}\pa{\brak{\sk{1}\cdot \grad_H B_t^{n,3}, B_t^{n,H}\cdot \grad\smkcomp{3}{2}}+ \brak{B_t^{n,H}\cdot \grad\skcomp{3}{2},\smk{1}\cdot \grad_H B_t^{n,3} }}dt+ dM_t^{n,3} ,
\end{align}
where \begin{align*}
    M_t^{n,3} = 2\sumkj\int_0^t  \brak{\skj\cdot \grad B_s^3 - B_s\cdot\grad \skjcomp{3}, B_s^{n,3}}dW_s^{k,j}.
\end{align*}
We can treat the different terms in the right hand side of \eqref{ito_l2_comp3} similarly to previous case.
Thanks to \autoref{Ito_Strat_correct_H-1} we obtain
\begin{align}\label{estimate_corr_L2}
\sumkj \pa{\|\skj\cdot\nabla B^{n,3}_t\|^2 + \|B^n_t\cdot\nabla \skjcomp{3}\|^2}& \leq \eta^n_T\lVert \nabla_H B^{n,3}_t\rVert^2+O(n^{-2})\lVert \nabla B^{n,3}_t\rVert^2+\left(\frac{\zeta^n_{H,2}}{C_{2,H}^2}+\frac{\zeta^n_{s,\gamma-2}}{C_{V}^2n^{\gamma-5}}\right)\lVert B^n_t\rVert^2. 
\end{align}
The term associated to the correlation of our noise can be treated easily by Holder and Young's inequality
\begin{align}\label{operator_correlationL2}
\frac{\rho \zeta^n_{H,2}} {C_{1,H}C_{2,H}}\brak{(B^{n,H}_t)^{\perp}, \nabla_H B_t^{n,3}}& \leq  \frac{\lvert\rho\rvert\zeta^n_{H,2}}{C_{1,H}C_{2,H}}\lVert B^{n,H}_t\rVert \lVert \nabla_H B_t^{n,3}\rVert\notag\\ & \leq \frac{\eta}{2} \lVert \nabla B^{n,3}\rVert^2+\frac{\lvert\rho\rvert^2\left(\zeta^n_{H,2}\right)^2}{2C_{1,H}^2C_{2,H}^2\eta}\lVert B^n_t\rVert^2,
\end{align}
\begin{align}\label{ito_strat_correlationL2}
\rho \sum_{\substack{k\in \Z^2_0\\ n\leq \lvert k\rvert\leq 2n}}\pa{\brak{\sk{1}\cdot \grad_H B_t^{n,3}, B_t^{n,H}\cdot \grad\smkcomp{3}{2}}}&\leq\lvert \rho\rvert \frac{\zeta_{H,3}^n\lVert \nabla_H B^{n,3}_t\rVert \lVert B^n_t\rVert}{C_{1,H}C_{2,H}n} \notag\\ & \leq \frac{\lVert \nabla B^{n,3}_t\rVert^2}{2n}+\frac{\lvert \rho\rvert^2\left(\zeta^n_{H,3}\right)^2\lVert B^n_t\rVert^2}{2C_{1,H}^2C_{2,H}^2n}   
\end{align}
and analogously for the other.
Therefore, integrating \eqref{ito_l2_comp3} between $0$ and $t$ and considering its expected value, thanks to \eqref{estimate_corr_L2}, \eqref{operator_correlationL2},\eqref{ito_strat_correlationL2} we obtain
\begin{align*}
   \expt{ \|B_t^{n,3}\|^2} &+\left( \eta-\frac{2}{n}+O(n^{-2})\right) \expt{\int_0^t\|\grad B_s^{n,3}\|^2ds}\notag \\ &\leq \norm{B^{n,3}_0}^2+   2 \left(\frac{\zeta^n_{H,2}}{C_{2,H}^2}+\frac{\zeta^n_{s,\gamma-2}}{C_{V}^2n^{\gamma-5}}+\frac{\lvert\rho\rvert^2\left(\zeta^n_{H,2}\right)^2}{2C_{1,H}^2C_{2,H}^2\eta}+O(n^{-1})\right)\expt{\int_0^T \lVert B^n_s\rVert^2 ds}.
\end{align*}
Therefore, if $n$ is large enough such that $\eta-\frac{2}{n}+O(n^{-2})>\frac{\eta}{2}$, thanks to \eqref{compactness space_h-1} we have immediately
\begin{align}\label{b3_no_sup_time}
 \sup_{t\in [0,T]}\expt{ \|B_t^{n,3}\|^2} &+\expt{\int_0^T\|\grad B_s^{n,3}\|^2ds}\lesssim \norm{B^{n,3}_0}^2+\norm{B^{n}_0}_{\mathbf{H}^{-1}}^2.   
\end{align}
Restarting again from \eqref{ito_l2_comp3} and considering the expected value of the supremum in time between $[0,T]$, thanks to \eqref{estimate_corr_L2}, \eqref{operator_correlationL2},\eqref{ito_strat_correlationL2}, we obtain for $n$ large enough such that $\eta-\frac{2}{n}+O(n^{-2})>\frac{\eta}{2}$
\begin{align}\label{suptcomp3prefinal}
 \expt{\sup_{t\in [0,T]} \|B_t^{n,3}\|^2}  &\leq \norm{B^{n,3}_0}^2+   2 \left(\frac{\zeta^n_{H,2}}{C_{2,H}^2}+\frac{\zeta^n_{s,\gamma-2}}{C_{V}^2n^{\gamma-5}}+\frac{\lvert\rho\rvert^2\left(\zeta^n_{H,2}\right)^2}{2C_{1,H}^2C_{2,H}^2\eta}+O(n^{-1})\right)\expt{\int_0^T \lVert B^n_s\rVert^2 ds}\notag\\ & +\expt{\sup_{t\in [0,T]}\lvert  M_t^{n,3}\rvert}.    
\end{align}
The first term can be treated easily thanks to \eqref{compactness space_h-1}. The second one is treated via Burkholder-Davis-Gundy inequality, \eqref{mirror symmetry product 0}, Young's inequality and \eqref{estimate_corr_L2} obtaining
\begin{align}\label{estimate_martingale}
\expt{\sup_{t\in [0,T]}\lvert  M_t^{n,3}\rvert}&\lesssim_{\lvert \rho\rvert}  \expt{\left(\sumkj\int_0^T  \brak{\skj\cdot \grad B_s^3 - B_s\cdot\grad \skjcomp{3}, B_s^{n,3}}^2ds\right)^{1/2}}\notag\\ & \leq   \expt{\sup_{t\in [0,T]}\norm{B^{n,3}_s}\left(\sumkj\int_0^T  \norm{\skj\cdot \grad B_s^3}+\norm{B_s\cdot\grad \skjcomp{3}}^2ds\right)^{1/2}}\notag\\ & \leq \frac{1}{2}\expt{\sup_{t\in [0,T]}\lVert  B_t^{n,3}\rVert}+C_{\lvert \rho\rvert,C_{1,H},C_{2,H},C_V,\gamma}\left(\expt{\int_0^T\lVert \nabla B^{n,3}_s\rVert^2 ds}+\expt{\int_0^T\lVert B^n_s\rVert^2 ds} \right).
\end{align}
Combining \eqref{suptcomp3prefinal}, \eqref{b3_no_sup_time}, \eqref{estimate_martingale} and \eqref{compactness space_h-1} we obtain relation \eqref{compactness space_l2}.
\end{proof}
 \begin{remark}\label{rmk negative eddy}
     Following the discussion in the introduction, let us comment on the assumptions on the noise that guarantee the validity of our energy estimate. In our framework, we have that the `mean field' magnetic operator is coercive in $\mathbf{H}^{-1}$ for any diffusivity $\eta >0$, meaning that there exists $C_{\eta},c_{\eta}>0$ such that for each $F\in \mathbf{L}^2$ and $n\in \N$
\begin{align*}
    \langle(-\Delta)^{-1}(\eta \Delta + \Lambda^n_\rho + \Lambda^n)F, F\rangle\le -C_{\eta}\|F\|^2 + c_{\eta}\|F\|^2_{\mathbf{H}^{-1}}.
\end{align*}
%meaning that we would expect nice behavior of the limit equation for any values of $\eta>0$.% 
However it is not clear that for each $\eta>0$ our equations are super-parabolic in $\mathbf{H}^{-1}$ in the sense of \cite{Flandoli_Book_95}, i.e.
\begin{equation}\label{stochastic coercivity}
 \langle(-\Delta)^{-1}(\eta \Delta + \Lambda^n_\rho + \Lambda^n)F, F\rangle + \sumkj\|\nabla \times (\skj \times F)\|^2_{\mathbf{H}^{-1}} \le - \tilde{C}_{\eta}\|F\|^2 + \tilde{c}_{\eta} \|F\|^2_{\mathbf{H}^{-1}}
\end{equation}
for possibly different $\tilde{C}_{\eta},\tilde{c}_{\eta}>0$ independent of $n$.
In \autoref{Ito_Strat_correct_H-1} we provided estimates on the quadratic variation term $\sumkj\|\nabla \times (\skj \times F)\|^2_{\mathbf{H}^{-1}}$ that allowed us to show that the solutions of the stochastic PDE, are under control only for large enough values of $\eta$ (with respect to the size of the noise). To the present, it is not clear if this limitation is only technical or if it is a manifestation of a phenomenon known in physics as `negative eddy viscosity', that is the appearance of an `effective' second order operator with a negative sign, which for us would mean the impossibility of establishing equation \eqref{stochastic coercivity} with a positive constant $\tilde{C}_{\eta}$. In such a case, we cannot expect the mean field equations \eqref{limit solution A3} to be the correct scaling limit model for a system developing negative eddy viscosity phenomena.
 \end{remark}
\section{Proof of \autoref{main Theorem}}\label{sec proof main thm}
Before exploiting the convergence properties of $B^n$, we start with a well-posedness result for our limit object $\overline{B}$ solution of \eqref{limit solution A3}.
\begin{definition}\label{def very weak solution linear}
    We say that $\overline{B}$ is a weak solution of equation \eqref{limit solution A3} if \begin{align*}
        \overline{B}\in C(0,T;\mathbf{H}^{-1})\cap L^2(0,T;\mathbf{L}^2)
\end{align*}
and for each $\phi \in \mathbf{H}^2$, and all $t\in [0,T]$, one has
\begin{align*}
    \langle {\overline{B}}_t,\phi\rangle-\langle {\overline{B}}_0,\phi\rangle &= \eta \int_0^t \langle {\overline{B}}_s,\Delta \phi\rangle\, ds +\int_0^t \langle {\overline{B}}_s, \Lambda_{\alpha,\beta,\gamma}\phi\rangle\, ds\\
    &\quad +\int_0^t \langle \mathcal{A}_{\rho}{\overline{B}}_s, \nabla\times\phi \rangle\, ds .
\end{align*}    
\end{definition}
Since the operator $\eta\Delta+\Lambda_{\alpha,\beta,\gamma}$ is uniformly elliptic and equation \eqref{def very weak solution linear} is linear, by standard theory of parabolic PDE we have existence and uniqueness of equation \eqref{def very weak solution linear} in the sense \autoref{def very weak solution linear}, see for example \cite[Chapter 3]{Lions_mag}, \cite[Chapter 5]{Flandoli_Book_95}.
\begin{theorem}\label{well_posed_limit}
    For each $B_0\in \mathbf{H}^{-1}$, there exists a unique solution of \eqref{limit solution A3} in the sense of \autoref{def very weak solution linear}. Moreover if $B_0\in \mathbf{L}^{2}$, then 
    \begin{align*}
        \overline{B}\in C(0,T;\mathbf{L}^2)\cap L^2(0,T;\mathbf{H}^1).
    \end{align*}
\end{theorem}
\begin{remark}\label{equation_for_B_3}
Arguing as in \autoref{general_duality_pairing} we obtain easily that, given $\overline{B_0}$ in $\mathbf{L}^2,$ then $\overline{B^3_t}$ satisfies 
\begin{align*}
    \langle \overline{B^{3}_t},\varphi\rangle-\langle \overline{B^{3}_0},\varphi\rangle&=\eta \int_0^t  \langle \overline{B^{3}_s},\Delta \varphi\rangle \, ds+\int_0^t  \langle \overline{B^{3}_s},\Lambda_{\alpha,\beta,\gamma}\varphi\rangle\, ds+\frac{\rho \zeta_{H,2}} {C_{1,H}C_{2,H}}\int_0^t  \langle {\overline{B^{H}_s}}^{\perp},\nabla_H\varphi\rangle\, ds
\end{align*}
for every $t\in [0,T]$ and $\varphi\in \Dot{H}^2$.
\end{remark}
Now we are ready to provide the proof of our main theorem. First we start with some notation.
Let us denote by $S^n,\ S$ the semigroups generated by $\eta \Delta+\Lambda^n$ and $\eta \Delta+\Lambda_{\alpha,\beta,\gamma}$. Since their infinitesimal generators have the form of $\mu_n(\partial_{11}+\partial_{22})+\nu_n\partial_{33},\ \mu_n\geq \nu_n>0$  we are in the framework of \autoref{operators_periodic_setting}. 
In particular, we recall that $\mu_n,\nu_n\geq \eta$ for each $n\in \N$.\\
Secondly, we denote by
\begin{align*}
Z^{n}_t&=\sumkj\int_0^t S^n(t-s)\nabla\times(\skj \times B^n_s)dW^{k,j}_s,  \\
H^{n}_t&=\sumkj \int_0^t S^n(t-s)\left(\skj\cdot \nabla B^{n,3}_t-B^n_t\cdot\nabla\skjcomp{3}\right)dW^{k,j}_s.
\end{align*}
The stochastic integrals above are well defined thanks to the regularity properties of $B^n_t$ and the fact that for each $n$ only a finite number of $\skj \neq 0$.
The first step in order to prove \autoref{main Theorem} is to rewrite $B^n_t,\ B^{n,3}_t$ in mild form. Indeed, recalling the definition of $\mathcal{A}^n_{\rho}$, see \autoref{lemma ito strat corrector}, the following lemma holds true.
\begin{lemma}\label{mild form}
$B^n_t$ and $B^{n,3}_t$ satisfy the mild formula
\begin{align}
\label{mild B^n}  
B^n_t&=S^n(t)B^n_0+\int_0^t S^n(t-s)\nabla \times (\mathcal{A}^n_{\rho}B^n_s)ds+ Z^{n}_t,\\ 
\label{mild B^n3}
   B^{n,3}_t&=S^n(t)B^{n,3}_0-\frac{\rho \zeta^n_{H,2}} {C_{1,H}C_{2,H}}\int_0^t S^{n}(t-s)\operatorname{div}_H\left((B^{n,H}_s)^{\perp}\right) ds+H^{n}_t,
\end{align}
the former seen as an equality in $\mathbf{H}^{-1}$ and the latter in $\Dot{L}^2$.
\end{lemma}
\begin{proof}
Let us take $\phi=\psi_{h,l},\ \psi_{h,l}=a_{h,l}e^{ih\cdot x}$ as test function in \eqref{weak formulation full system} obtaining 
\begin{align}
     \langle B^n_t,\psi_{h,l}\rangle
    &=  \langle B^n_0,\psi_{h,l}\rangle-\left(\mu_n\left(h_1^2+h_2^2\right)+\nu_n h_3^2\right)\int_0^t  \langle B^n_s,\psi_{h,l}\rangle \, ds\notag\\ &+
    \int_0^t \langle \nabla\times(\mathcal{A}^n_{\rho}B^n_s),\psi_{h,l}\rangle \, ds+\sumkj \int_0^t \langle \nabla\times(\skj\times B^n_s),\psi_{h,l}\rangle \, dW^{k,j}_s
    \label{mild form A3 step 1}.
\end{align}
Therefore, applying It\^o formula to $e^{t\left(\mu_n\left(h_1^2+h_2^2\right)+\nu_n h_3^2\right)}\langle B^n_t,\psi_{h,l}\rangle$ we have
\begin{align}\label{fourier mild form barA3}
\langle B^n_t,\psi_{h,l}\rangle&=e^{-t\left(\mu_n\left(h_1^2+h_2^2\right)+\nu_n h_3^2\right)}   \langle B^n_0,\psi_{h,l}\rangle+\int_0^t e^{-(t-s)\left(\mu_n\left(h_1^2+h_2^2\right)+\nu_n h_3^2\right)}\langle \nabla\times(\mathcal{A}^n_{\rho}B^n_s),\psi_{h,l}\rangle   ds\notag\\ &+\sumkj\int_0^t e^{-(t-s)\left(\mu_n\left(h_1^2+h_2^2\right)+\nu_n h_3^2\right)}\langle \nabla\times(\skj\times B^n_s),\psi_{h,l}\rangle   dW^{k,j}_s\quad \mathbb{P}-a.s.\quad \forall t\in [0,T].
\end{align}
We can find $\Gamma\subseteq \Omega$ of full probability such that the above equality holds for all $t\in [0,T]$ and all $h\in \Z^3_0,\ l\in \{1,2\}$. But this is exactly \eqref{mild B^n} written in Fourier modes. The proof of \eqref{mild B^n3} is analogous, just replacing relation \eqref{weak formulation full system} with \autoref{general_duality_pairing} and choosing as test function $\phi=(0,0,\psi_h)^t,\ \psi_h=e^{ih\cdot x},\ h\in\Z^3_0$. We omit the easy details.
\end{proof}
Secondly we need to study some properties of the stochastic convolutions $Z^{n}_t,H^{n}_t$. \autoref{a_priori_bound_stocastic_conv} and \autoref{convergence_stocastic_conv} below are the analogous of \cite[Lemma 2.5]{flandoli2021quantitative} in our framework. In particular it is important to point out that \cite[Assumption 2.4]{flandoli2021quantitative} is false in our case and we have to deal also with the stochastic stretching, which was neglected in previous results. In order to reach our plan, we start with a uniform bound on our coefficients in a distributional norm.
\begin{lemma}\label{lemma:coefficients}
    Assuming either \autoref{HP noise 1} or \autoref{HP noise isotropo}, for each $\delta>0$
    \begin{align}\label{estimate_no convergence}
    \expt{\sup_{t\in [0,T]}\sumkj \lVert P[Q[\skj\times B^{n}_t]]\rVert_{\mathbf{H}^{-1-\frac{\delta}{2}}}^2}&\lesssim \lVert B_0^n\rVert_{\mathbf{H}^{-1}}^2.
    \end{align}
\end{lemma}
\begin{proof}
Arguing as in \autoref{Ito_Strat_correct_H-1}, if we denote by \begin{align*}
        b^{n,h,l}_t:=\langle \Bt{t}, a_{h,l}e^{-ih\cdot x}\rangle,
    \end{align*}
    then
    \begin{align*}
    \sigma^n_{k,j}\times \Bt{t}&=\tkj\sum_{\substack{h\in \Z^3_0,\\ l\in \{1,2\}}}   b^{n,h,l}_t a_{k,j}\times a_{h,l}e^{i(h+k)\cdot x}.
    \end{align*} 
    Thanks to \autoref{prop:compact_space} we already know that
    \begin{align}\label{H-1fouriermode}
        \expt{\sup_{t\in [0,T]}\sum_{\substack{h\in \Z^3_0,\\ l\in \{1,2\}}} \frac{\lvert b^{n,h,l}_t\rvert^2}{\lvert h\rvert^2}}\lesssim \lVert \Bt{0}\rVert_{\mathbf{H}^{-1}}^2.
    \end{align}     
For each $k\in \Z^3_0$ we have
\begin{align*}
   \norm{P[Q[\skj\times \Bt{t}]]}_{\mathbf{H}^{-1-\frac{\delta}{2}}}^2&\lesssim\sum_{\substack{h\in \Z_0^3\\ h\neq -k,\ l\in \{1,2\}}} \frac{\lvert\tkj\rvert^2}{\lvert h+k\rvert^{2+\delta}} \lvert b^{n,h,l}_t\rvert^2.
\end{align*}
We are left to study
\begin{align}\label{prefinal_coeff_stoch_conv_1}
    \expt{\sup_{t\in [0,T]}\sum_{\substack{k\in \Z^3_0\\ n\leq \lvert k\rvert\leq 2n\\ j\in \{1,2\}}}\sum_{\substack{h\in \Z_0^3\\ h\neq -k,\\ l\in \{1,2\}}} \frac{\lvert \tkj b^{n,h,l}_t\rvert^2}{\lvert h+k\rvert^{2+\delta}} }&=\expt{\sup_{t\in [0,T]}\sum_{\substack{h\in \Z_0^3\\  l\in \{1,2\}}}\frac{\lvert b^{n,h,l}_t\rvert^2}{\lvert h\rvert^2} \sum_{\substack{k\in \Z^3_0\\h\neq -k,\\ n\leq \lvert k\rvert\leq 2n\\ j\in \{1,2\} }}\frac{\lvert \tkj h\rvert^2}{\lvert h+k\rvert^{2+\delta}} }. 
\end{align}
Due to relation \eqref{H-1fouriermode}, it is enough to show that 
\begin{align*}
\sum_{\substack{k\in \Z^3_0\\h\neq -k,\\ n\leq \lvert k\rvert\leq 2n\\ j\in \{1,2\}}}\frac{\lvert \tkj\rvert^2\lvert h\rvert^2}{\lvert h+k\rvert^{2+\delta}}\lesssim 1.
\end{align*}
uniformly in $h\in \Z^3_0,\ l\in \{1,2\}$. By splitting the summation domain, we have
\begin{align}\label{ineq_splitting_domain}
\sum_{\substack{k\in \Z^3_0\\h\neq -k,\\ n\leq \lvert k\rvert\leq 2n\\ j\in \{1,2\}}}\frac{\lvert \tkj\rvert^2\lvert h\rvert^2}{\lvert h+k\rvert^{2+\delta}}&\leq \sum_{\substack{k\in \Z^3_0\\ 1\leq \lvert k\rvert \leq \frac{\lvert h\rvert}{2},\\ n\leq \lvert k\rvert\leq 2n\\ j\in \{1,2\}}}\frac{\lvert \tkj\rvert^2\lvert h\rvert^2}{\lvert h+k\rvert^{2+\delta}} +\sum_{\substack{k\in \Z^3_0\\ 1\leq \lvert h+k\rvert \leq \frac{\lvert h\rvert}{2},\\ n\leq \lvert k\rvert\leq 2n\\ j\in \{1,2\}}}\frac{\lvert \tkj\rvert^2\lvert h\rvert^2}{\lvert h+k\rvert^{2+\delta}}+\sum_{\substack{k\in \Z^3_0\\  \lvert h+k\rvert \geq  \frac{\lvert h\rvert}{2},\\ \frac{\lvert h\rvert}{2}\vee n\leq \lvert k\rvert\leq 2n\\ j\in \{1,2\}}}\frac{\lvert \tkj\rvert^2\lvert h\rvert^2}{\lvert h+k\rvert^{2+\delta}}\notag \\ & \lesssim \sum_{\substack{k\in \Z^3_0\\ 1\leq \lvert k\rvert \leq \frac{\lvert h\rvert}{2},\\ n\leq \lvert k\rvert\leq 2n\\ j\in \{1,2\}}}\lvert \tkj\rvert^2+\sum_{\substack{k\in \Z^3_0\\ 1\leq \lvert h+k\rvert \leq \frac{\lvert h\rvert}{2},\\ n\leq \lvert k\rvert\leq 2n\\ j\in \{1,2\}}}\frac{\lvert \tkj\rvert^2\lvert h\rvert^2}{\lvert h+k\rvert^{2+\delta}}+\sum_{\substack{k\in \Z^3_0\\  \lvert h+k\rvert \geq  \frac{\lvert h\rvert}{2},\\ \frac{\lvert h\rvert}{2}\vee n\leq \lvert k\rvert\leq 2n\\ j\in \{1,2\}}}\lvert \tkj\rvert^2    
\end{align}
We start considering the case of \autoref{HP noise isotropo} which is the easier one. In this case, due to \eqref{ineq_splitting_domain}
\begin{align*}
\sum_{\substack{k\in \Z^3_0\\h\neq -k,\\ n\leq \lvert k\rvert\leq 2n\\ j\in \{1,2\}}}\frac{\lvert \tkj\rvert^2\lvert h\rvert^2}{\lvert h+k\rvert^{2+\delta}} & \lesssim \sum_{\substack{k\in \Z^3_0\\ n\leq \lvert k\rvert \leq 2n}}\frac{1}{\lvert k\rvert^3}+\sum_{\substack{k\in \Z^3_0\\ 1\leq \lvert h+k\rvert \leq \frac{\lvert h\rvert}{2},\\ n\leq \lvert k\rvert\leq 2n}}\frac{1}{\lvert k\rvert\lvert h+k\rvert^{2+\delta}}\\ & \lesssim 1+\left(\sum_{\substack{k\in \Z^3_0\\ n\leq \lvert k\rvert\leq 2n}}\frac{1}{\lvert k\rvert^3}\right)^{1/3}\left(\sum_{\substack{k\in \Z^3_0\\ 1\leq \lvert h+k\rvert \leq \frac{\lvert h\rvert}{2}}}\frac{1}{\lvert h+k\rvert^{3+\frac{3}{2}\delta}}\right)^{2/3}\lesssim_{\delta} 1,    
\end{align*}
where the first inequality follows from the fact that $\lvert k\rvert\geq \lvert (k+h)-h\rvert\geq \frac{\lvert h\rvert}{2}$ if $1\leq \lvert h+k\rvert \leq \frac{\lvert h\rvert}{2}$ and the second by H\"older's inequality. This completes the proof of \eqref{estimate_no convergence} in case of \autoref{HP noise isotropo}. In case of \autoref{HP noise 1} we argue similarly. Due to the summability properties of the $\tkj$, it is enough to consider the second term in the right hand side of \eqref{ineq_splitting_domain}, i.e.
\begin{align}\label{sum_to_control}
    \sum_{\substack{k\in \Z^3_0\\ 1\leq \lvert h+k\rvert \leq \frac{\lvert h\rvert}{2},\\ n\leq \lvert k\rvert\leq 2n\\ j\in \{1,2\}}}\frac{\lvert \tkj\rvert^2\lvert h\rvert^2}{\lvert h+k\rvert^{2+\delta}}& =  \sum_{\substack{k\in \Z^2_0\\ 1\leq \lvert h+k\rvert \leq \frac{\lvert h\rvert}{2},\\ n\leq \lvert k\rvert\leq 2n}}\frac{\lvert h\rvert^2}{\lvert k\rvert^\alpha\lvert h+k\rvert^{2+\delta}}+\sum_{\substack{k\in \Z^2_0\\ 1\leq \lvert h+k\rvert \leq \frac{\lvert h\rvert}{2},\\ n\leq \lvert k\rvert\leq 2n}}\frac{\lvert h\rvert^2}{\lvert k\rvert^{\beta}\lvert h+k\rvert^{2+\delta}}+\sum_{\substack{k\in \Z^3_0\\ 1\leq \lvert h+k\rvert \leq \frac{\lvert h\rvert}{2}\\ k_3\neq 0,\\ n\leq \lvert k\rvert\leq 2n}}\frac{\lvert h\rvert^2}{\lvert k\rvert^{\gamma}\lvert h+k\rvert^{2+\delta}}.
\end{align}
Since $\gamma>3$ the third term is a simplified version of the one treated before. We are left to study the other two. It is clear that if we are able to control  $\sum_{\substack{k\in \Z^2_0\\ 1\leq \lvert h+k\rvert \leq \frac{\lvert h\rvert}{2},\\ n\leq \lvert k\rvert\leq 2n}}\frac{\lvert h\rvert^2}{\lvert k\rvert^2\lvert h+k\rvert^{2+\delta}}$ we can control relation \eqref{sum_to_control}. Therefore we study only this case. We observe that by triangular inequality, the constraints $1\leq \lvert h+k\rvert \leq \frac{\lvert h\rvert}{2}, n\leq \lvert k\rvert\leq 2n$ imply
\begin{align*}
    \lvert h\rvert=\lvert (h+k)+(-k)\rvert \leq \lvert k\rvert+\frac{\lvert h\rvert}{2}\leq 2n+\frac{\lvert h\rvert}{2}. 
\end{align*}
As a consequence
\begin{align*}
\sum_{\substack{k\in \Z^2_0\\ 1\leq \lvert h+k\rvert \leq \frac{\lvert h\rvert}{2},\\ n\leq \lvert k\rvert\leq 2n}}\frac{\lvert h\rvert^2}{\lvert k\rvert^2\lvert h+k\rvert^{2+\delta}}\lesssim    \sum_{\substack{z\in h+\Z^2_0\\ 1\leq \lvert z\rvert \leq 2n}}\frac{1}{\lvert z\rvert^{2+\delta}}\lesssim_{\delta} 1.
\end{align*}
This concludes the proof of relation \eqref{estimate_no convergence}. 
\end{proof}
\autoref{lemma:coefficients} was a necessary toolds in order to prove the following estimates on $Z^n_t,\ H^n_t$.
\begin{lemma}\label{a_priori_bound_stocastic_conv}
For each $\delta>0$
\begin{align}\label{bound_weak_conv}
    \sup_{t\in [0,T]}\expt{\norm{Z^{n}_t}^2_{\mathbf{H}^{-1-\delta}}}\lesssim \lVert B_0^n\rVert_{\mathbf{H}^{-1}}^2.
\end{align}
In case of \autoref{HP noise 2} we have also
\begin{align}\label{bound_weak_conv_l2}
    \sup_{t\in [0,T]}\expt{\norm{H^{n}_t}^2_{\Dot{H}^{-\delta}}}\lesssim \lVert B^n_0\rVert_{\mathbf{H}^{-1}}^2+\lVert B^{n,3}_0\rVert^2.
\end{align}    
\end{lemma}
\begin{proof}
By Burkholder-Davis-Gundy inequality and the regularization properties of the semigroup, see \autoref{Properties semigroup}, we have
\begin{align*}
\expt{\norm{Z^{n}_t}^2_{\mathbf{H}^{-1-\delta}}}&\lesssim_{\lvert \rho\rvert}\sumkj \expt{\int_0^t \norm{ S^n(t-s)\nabla\times(\skj\times B^n_s)}_{\mathbf{H}^{-1-\delta}}^2  ds}\\ & \lesssim_{\eta}  \expt{\int_0^t \frac{1}{(t-s)^{1-\frac{\delta}{2}}}\norm{ \nabla\times(\skj\times B^n_s)}_{\mathbf{H}^{-2-\frac{\delta}{2}}}^2  ds}\\  & \lesssim  \expt{\int_0^t \frac{1}{(t-s)^{1-\frac{\delta}{2}}}\norm{ P[Q[\skj\times B^n_s]]}_{\mathbf{H}^{-1-\frac{\delta}{2}}}^2  ds}\\ & \lesssim_{\delta} \lVert B_0^n\rVert_{\mathbf{H}^{-1}}^2
\end{align*}
where the last inequality follows from \eqref{estimate_no convergence} and is uniform for $t\in [0,T].$ Assuming \autoref{HP noise 2}, similarly we can prove \eqref{bound_weak_conv_l2}. Again, by Burkholder-Davis-Gundy inequality and the regularization properties of the semigroup, \autoref{Properties semigroup}, we get
\begin{align*}
\expt{\norm{H^{n}_t}^2_{\Dot{H}^{-\delta}}}&\lesssim_{\lvert \rho\rvert}\sumkj \expt{\int_0^t \norm{S^n(t-s)\left[\skj\cdot\nabla B^{n,3}_s\right]-S^n(t-s)\left[B^n_s\cdot\nabla \skjcomp{3}\right]}_{\Dot{H}^{-\delta}}^2ds}  \\ & \lesssim_{\eta}  
\sumkj \expt{\int_0^t \frac{1}{(t-s)^{1-\delta}}\norm{\skj\cdot\nabla B^{n,3}_s}_{\Dot{H}^{-1}}^2ds} +\sumkj \expt{\int_0^t \norm{B^n_s\cdot\nabla \skjcomp{3}}^2ds}\\ & \lesssim \sumkj \expt{\int_0^t \frac{1}{(t-s)^{1-\delta}}\norm{\skj B^{n,3}_s}^2ds} +\sumkj \expt{\int_0^t \norm{B^n_s\cdot\nabla  \skjcomp{3}}^2ds}.
\end{align*}
Both the two terms can be treated easily. Indeed,
\begin{align*}
    \sumkj \expt{\int_0^t \frac{1}{(t-s)^{1-\delta}}\norm{\skj B^{n,3}_s}^2ds}&\leq \sumkj \lvert \tkj\rvert^2 \expt{\int_0^t \frac{1 }{(t-s)^{1-\delta}}\sup_{s\in [0,T]}\norm{B^{n,3}_s}^2 ds}\\ &\lesssim_{\delta} \lVert B^n_0\rVert_{\mathbf{H}^{-1}}^2+\lVert B^{n,3}_0\rVert^2.
\end{align*}
due to \autoref{prop:compact_space}.
Similarly
\begin{align*}
    \sumkj \expt{\int_0^t \norm{B^n_s\cdot\nabla  \skjcomp{3}}^2ds}&\leq \left(\sum_{\substack{k\in \Z^2_0\\ n\leq \lvert k\rvert \leq 2n}}\frac{1}{\lvert k\rvert^{\beta-2}}+\sum_{\substack{k\in \Z^3_0\\ n\leq \lvert k\rvert \leq 2n}}\frac{1}{\lvert k\rvert^{\gamma-2}}\right)\expt{\int_0^T \lVert B^n_s\rVert^2 ds}\\ & \lesssim \lVert B^n_0\rVert_{\mathbf{H}^{-1}}^2,
\end{align*}
again by \autoref{prop:compact_space}.
\end{proof}
\begin{lemma}\label{convergence_stocastic_conv}
Assuming either \autoref{HP noise 1}, \autoref{HP noise isotropo} then for each $\delta>0$
\begin{align}\label{estimate_weak_conv}
    \sup_{t\in [0,T]}\expt{\norm{Z^{n}_t}^2_{\mathbf{H}^{-4-\frac{3}{2}\delta}}}\lesssim \frac{\norm{B^n_0}_{\mathbf{H}^{-1}}^2}{ n^{\alpha\wedge\beta\wedge\gamma}}.
\end{align}
In case of \autoref{HP noise 2} we have also
\begin{align}\label{estimate_weak_conv_l2}
    \sup_{t\in [0,T]}\expt{\norm{H^{n}_t}^2_{\Dot{H}^{-\frac{3}{2}-\frac{3}{2}\delta}}}\lesssim \frac{\norm{B^{n,3}_0}^2+\norm{B^n_0}_{\mathbf{H}^{-1}}^2}{ n^{2}}.
\end{align}
\end{lemma}
\begin{proof}
Let us consider for a smooth divergence free  vector field $\psi$ 
\begin{align*}
    \expt{\langle Z^n_t,\psi\rangle^2}=\expt{ \left(\sumkj\int_0^t \langle S^n(t-s)\nabla\times(\skj\times B^n_s),\psi\rangle dW^{k,j}_s\right)^2}.
\end{align*}
Therefore, by Burkholder-Davis-Gundy inequality, Sobolev embedding theorem and the contractivity property of our semigroup we obtain
\begin{align*}
\expt{\langle Z^n_t,\psi\rangle^2}&\lesssim_{\lvert \rho\rvert}\expt{\sumkj\int_0^t \langle S^n(t-s)\nabla\times(\skj\times B^n_s),\psi\rangle^2 ds}\\ & = \expt{\sumkj \lvert \tkj\rvert^2\int_0^t \langle a_{k,j}e^{ik\cdot x},B^{n}_s\times\left(\nabla\times S^n(t-s)\psi\right)\rangle^2 ds}\\ & \lesssim \frac{1}{n^{\alpha\wedge\beta\wedge\gamma}}\expt{\sumkj \int_0^t \langle a_{k,j}e^{ik\cdot x},B^{n}_s\times\left(\nabla\times S^n(t-s)\psi\right)\rangle^2 ds}\\ & \leq \frac{1}{n^{\alpha\wedge\beta\wedge\gamma}}\expt{\int_0^t \lVert B^n_s\times \left(\nabla\times S^n(t-s)\psi\right)\rVert^2 ds } \\ & \lesssim_{\eta,\delta} \frac{1}{n^{\alpha\wedge\beta\wedge\gamma}}\lVert \psi\rVert_{H^{\frac{5}{2}+\delta}}^2\expt{\int_0^T \lVert B^n_s\rVert^2 ds },
\end{align*}
uniformly in $t\in [0,T]$.
Therefore, choosing $\psi=a_{h,l}e^{ih\cdot x},\ h\in Z^3_0,\ l\in \{1,2\}$ we obtain, due to \autoref{prop:compact_space}
\begin{align}\label{ineq_orth_base}
 \expt{\langle Z^n_t,a_{h,l}e^{ih\cdot x}\rangle^2}\lesssim \frac{1}{n^{\alpha\wedge\beta\wedge\gamma}}\lvert h\rvert^{5+2\delta}{ \lVert B^n_0\rVert_{\mathbf{H}^{-1}}^2  }.     
\end{align}
Relation \eqref{estimate_weak_conv} then follows from \eqref{ineq_orth_base}. Indeed, for each $t\in [0,T]$
\begin{align*}
 \expt{\norm{Z^{n}_t}^2_{\mathbf{H}^{-4-\frac{3}{2}\delta}}}&=\sum_{\substack{h\in \Z^3_0\\ l\in \{1,2\}}}\frac{\expt{\langle Z^n_t,a_{h,l}e^{ih\cdot x}\rangle^2} }{\lvert h\rvert^{8+3\delta}} \\ & \lesssim  \frac{1}{n^{\alpha\wedge\beta\wedge\gamma}}\sum_{h\in \Z^3_0}\frac{\lVert B^n_0\rVert_{\mathbf{H}^{-1}}^2}{\lvert h\rvert^{3+\delta}}\\ & \lesssim_{\delta}\frac{\lVert B^n_0\rVert_{\mathbf{H}^{-1}}^2}{ n^{\alpha\wedge\beta\wedge\gamma}}.
\end{align*}
In order to get \eqref{estimate_weak_conv_l2} we argue slightly differently.
By Burkholder-Davis-Gundy inequality and the regularization properties of the semigroup, see \autoref{Properties semigroup}, we have
\begin{align*}
\expt{\norm{H^{n}_t}^2_{\mathbf{H}^{-\frac{3}{2}-\frac{3}{2}\delta}}}&\lesssim_{\lvert \rho\rvert}\sumkj \expt{\int_0^t \norm{S^n(t-s)\left[\skj\cdot\nabla B^{n,3}_s\right]-S^n(t-s)\left[B^n_s\cdot\nabla \skjcomp{3}\right]}_{\Dot{H}^{-\frac{3}{2}-\frac{3}{2}\delta}}^2ds}  \\ & \lesssim  
\sumkj \expt{\int_0^t \frac{1}{(t-s)^{1-\delta}}\norm{\skj\cdot\nabla B^{n,3}_s}_{\Dot{H}^{-\frac{5}{2}-\frac{1}{2}\delta}}^2ds}\\ & +\sumkj \expt{\int_0^t \norm{B^n_s\cdot\nabla \skjcomp{3}}_{\Dot{H}^{-\frac{3}{2}-\frac{3}{2}\delta}}^2ds}\\ & \lesssim \sumkj \expt{\int_0^t \frac{1}{(t-s)^{1-\delta}}\norm{\skj B^{n,3}_s}_{\Dot{H}^{-\frac{3}{2}-\frac{1}{2}\delta}}^2ds} +\sumkj \expt{\int_0^t \norm{B^n_s  \skjcomp{3}}^2ds}.
\end{align*}
The second term can be analyzed easily by \autoref{prop:compact_space} obtaining
\begin{align*}
\sumkj \expt{\int_0^t \norm{B^n_s  \skjcomp{3}}^2ds}&\lesssim \left(\sum_{\substack{k\in \Z^2_0\\ n\leq \lvert k\rvert \leq 2n}}\frac{1}{\lvert k\rvert^4}+\sum_{\substack{k\in \Z^3_0\\ n\leq \lvert k\rvert \leq 2n}}\frac{1}{\lvert k\rvert^\gamma}  \right)\norm{B^n_0}_{\mathbf{H}^{-1}}^2\\ & \lesssim   \frac{\norm{B^n_0}_{\mathbf{H}^{-1}}^2}{n^{2}}.
\end{align*}
In order to treat the other one we need to argue in a more precise way. Indeed, by definition of the Sobolev norm on the 3D torus, we have
\begin{align*}
    \sumkj \expt{\int_0^t \frac{1}{(t-s)^{1-\delta}}\norm{\skj B^{n,3}_s}_{\Dot{H}^{-\frac{3}{2}-\frac{1}{2}\delta}}^2ds} & \lesssim \frac{1}{n^2}\sum_{k\in \Z^3_0}\expt{\int_0^t \frac{1}{(t-s)^{1-\delta}}\norm{e^{ik\cdot x}B^{n,3}_s}_{\Dot{H}^{-\frac{3}{2}-\frac{1}{2}\delta}}^2ds}\\ & = \frac{1}{n^2}\sum_{h\in \Z^3_0}\frac{1}{\lvert h\rvert^{3+\delta}} \sum_{k\in \Z^3_0} \expt{\int_0^t \frac{1}{(t-s)^{1-\delta}}\langle e^{i(k-h)\cdot x},B^{n,3}_s\rangle^2ds}\\ & \leq \frac{1}{n^2}\sum_{h\in \Z^3_0}\frac{1}{\lvert h\rvert^{3+\delta}} \expt{\int_0^t \frac{1}{(t-s)^{1-\delta}} \sup_{t\in  [0,T]}\norm{ B^{n,3}_t}^2 ds}\\ & \lesssim_{\delta} \frac{\norm{B^n_0}_{\mathbf{H}^{-1}}^2+\norm{B^{n,3}_0}^2}{ n^2}
\end{align*}
uniformly in $t\in[0,T]$. This concludes the proof of \eqref{estimate_weak_conv_l2}.
\end{proof}
Combining \autoref{a_priori_bound_stocastic_conv} and \autoref{convergence_stocastic_conv}, by interpolation we get the following result.
\begin{corollary}\label{cor_stoch_conv}
Assuming either \autoref{HP noise 1} or \autoref{HP noise isotropo}, for each $\theta \in (0,3]$ and $\delta\in (0,\theta]$ we have
\begin{align*}
    \sup_{t\in [0,T]}\expt{\lVert Z^n_t\rVert_{\mathbf{H}^{-1-\theta}}^2}\lesssim\frac{1}{n^{\frac{(\alpha\wedge\beta\wedge\gamma)(\theta-\delta)}{3+\frac{\delta}{2}}}}\norm{B^n_0}_{\mathbf{H}^{-1}}^2.
\end{align*}
In case of \autoref{HP noise 2} we have also for each $\vartheta\in (0,\frac{3}{2}]$ and $\delta\in (0,\vartheta]$ \begin{align*}
    \sup_{t\in [0,T]}\expt{\norm{H^{n}_t}^2_{\Dot{H}^{-\vartheta}}}\lesssim \frac{\norm{B^{n,3}_0}^2+\norm{B^n_0}_{\mathbf{H}^{-1}}^2}{ n^{\frac{4(\vartheta-\delta)}{3}}}.
\end{align*}
\end{corollary}

Moreover, we recall that by classical theory of evolution equations, see for example \cite{Lunardi}, or arguing as in \autoref{mild form}, the unique weak solutions of \eqref{limit solution A3} and its third component can be written in mild form as
\begin{align*}
\overline{B_t}=S(t)\overline{B_0}+\int_0^t S(t-s) \nabla \times (\mathcal{A}_{\rho}\overline{B_s})ds, \quad \overline{B_t^{3}}=S(t-s)\overline{B^3_0}-\frac{\rho \zeta_{H,2}} {C_{1,H}C_{2,H}}\int_0^t S(t-s)\operatorname{div}_H\left(\overline{B^{H}_s}^{\perp}\right) ds.
\end{align*}
Now we introduce some intermediate vector fields between $B^n_t$ and $\overline{B_t}$. Let $\widehat{B_t^{n}}$ the unique weak solution of the linear system
\begin{align}
\label{intermediate systems B}
&\begin{cases}
\partial_t \widehat{B^{n}_t}&=(\eta\Delta+\Lambda^n)\widehat{B^{n}_t}+\nabla\times(\mathcal{A}^n_{\rho}\widehat{B^{n}_t})\quad x\in \T^3,\ t\in (0,T)\\
\widehat{B^{n}_t}|_{t=0}&=B^n_0.   
\end{cases}    
\end{align}
Again, by classical theory of evolution equations, see for example \cite{Lions_mag, Flandoli_Book_95}, previous equation is well-posed either in $\mathbf{H}^{-1}$ and $\mathbf{L}^2$. Moreover the unique weak solution and its third component can be written in mild form as
\begin{align*}
\widehat{B_t^{n}}&=S^n(t-s)B^n_0+\int_0^t S^n(t-s)\nabla\times(\mathcal{A}^n_{\rho}\widehat{B^{n}_s})ds, \\ \widehat{B_t^{3,n}}&=S^n(t-s)B^{n,3}_0-\frac{\rho \zeta^n_{H,2}} {C_{1,H}C_{2,H}}\int_0^t S^n(t-s)\operatorname{div}_H\left(\widehat{B^{n,H}_s}^{\perp}\right) ds.
\end{align*}
Now we prove that $\widehat{B^{n}_t}$
and $\overline{B_t}$ are close. We recall that the function $\chi$ introduced in \autoref{notation sec} satisfies 
\begin{align*}
    0<\chi(\alpha,\beta,\gamma)\leq 1
\end{align*}
for each choice of the parameters in the range described by \autoref{HP noise 1}, \autoref{HP noise 2} and \autoref{HP noise isotropo}.
\begin{lemma}\label{preliminary convergence}
Under the same assumptions of \autoref{main Theorem}, for each $\theta \in (0,2)$  we have
\begin{align}\label{intermediate_convergence_B}
\sup_{t\in [0,T]}\norm{\widehat{B^{n}_t}-\overline{B_t}}_{\Dot{H}^{-1-\theta}}&\lesssim_{\eta,T} \frac{1}{n^{\theta \chi(\alpha,\beta,\gamma)/2}}+\norm{\overline{B_0}-B^n_0}_{\mathbf{H}^{-1-\theta}}.
       \end{align}
In case of \autoref{HP noise 2} we have also for each $\vartheta\in (\theta,\frac{3}{2}]$
\begin{align}\label{intermediate_convergence_B3}
\sup_{t\in [0,T]}\norm{\widehat{B^{n,3}_t}-{\overline{B^3_t}}}_{\Dot{H}^{-\vartheta}}&\lesssim_{\eta,T} \frac{1}{(\vartheta-\theta)n^{\theta/2}}+\norm{\overline{B^3_0}-B^{n,3}_0}_{\Dot{H}^{-\vartheta}}+\frac{\norm{\overline{B_0}-B^n_0}_{\mathbf{H}^{-1-\theta}}}{\vartheta-\theta}.
\end{align}
\end{lemma}
\begin{proof}
 First let us observe that by assumptions in case of \autoref{HP noise 1} and \autoref{HP noise isotropo} (resp. \autoref{HP noise 2})
\begin{align*}
B^n_0\rightharpoonup \overline{B_0} \text{ in }\mathbf{H}^{-1}\quad (\text{resp. } B^n_0\rightharpoonup \overline{B_0} \text{ in }\mathbf{H}^{-1},\ B^{n,3}_0\rightharpoonup \overline{B^{3}_0} \text{ in }\Dot L^2(\T^3)) 
\end{align*}
 we have in particular that the family $\left\{\norm{B^n_0}_{\mathbf{H}^{-1}}^2\right\}_{n\in \N}$ (resp. $\left\{\norm{B^n_0}_{\mathbf{H}^{-1}}^2\right\}_{n\in \N},\ \left\{\norm{B^{n,3}_0}^2\right\}_{n\in \N}$) is bounded.
Moreover, $\mathcal{A}^n_{\rho}=\mathcal{A}_{\rho}+O(\frac{1}{n})$, then in particular we have
\begin{align}\label{uniform bound matrix}
    \sup_{n\in\N}\norm{\mathcal{A}^n_{\rho}}_{HS}<+\infty.
\end{align}
Lastly, by simple energy estimates on \eqref{intermediate systems B} we have for each $n\in \N$
\begin{align}\label{uniform bound auxiliary A}
    \sup_{t\in [0,T]}\norm{\widehat{B^{n}_t}}_{\mathbf{H}^{-1}}^2+\int_0^T \norm{\widehat{B^{n}_s}}^2 ds\lesssim_{\eta} \sup_{n\in\N}\norm{\widehat{B^{n}_0}}_{\mathbf{H}^{-1}}^2<+\infty,\\
    \label{uniform_bound_third_component}
    \sup_{t\in [0,T]}\norm{\widehat{B^{n,3}_t}}^2+\int_0^T \norm{\nabla\widehat{B^{n,3}_s}}^2 ds\lesssim_{\eta} \sup_{n\in\N}\norm{\widehat{B^{n}_0}}_{\mathbf{H}^{-1}}^2+\sup_{n\in\N}\norm{\widehat{B^{n,3}_0}}^2<+\infty.
\end{align}
The convergence of $\widehat{B^{n}_t}$ to $\overline{B_t}$ then follows by triangle inequality, \autoref{Properties semigroup}, \autoref{convergence_operators} and the uniform bounds on the initial conditions. Indeed for each $t\in [0,T]$ it holds
\begin{align*}
\norm{\widehat{B^{n}_t}-\overline{B_t}}_{\Dot{H}^{-1-\theta}}& \leq \left(\norm{\left(S^n(t)-S(t)\right)\overline{B_0}}_{\mathbf{H}^{-1-\theta}}+\norm{S^n(t)\left(B^n_0-\overline{B_0}\right)}_{\mathbf{H}^{-1-\theta}}\right)\\& +\norm{\int_0^t \left(S^n(t-s)-S(t-s)\right) \nabla\times\left(\mathcal{A}^n_{\rho}\widehat{B^{n}_s}\right) ds}_{\mathbf{H}^{-1-\theta}}  \\ & + \norm{\int_0^t S(t-s)\nabla\times\left(\left(\mathcal{A}^n_{\rho}-\mathcal{A}_{\rho}\right)\widehat{B^{n}_s}\right) ds}_{\mathbf{H}^{-1-\theta}}\\ & + \norm{\int_0^t S(t-s)\nabla\times\left(\mathcal{A}_{\rho}\left(\widehat{B^{n}_s}-\overline{B_s}\right)\right) ds}_{\mathbf{H}^{-1-\theta}}=I_1+I_2+I_3+I_4. 
\end{align*}
$I_1$ can be treated by \autoref{convergence_operators} :
\begin{align}\label{intermediate B step 1}
I_1\lesssim_{\eta,T} \frac{1}{n^{\theta \chi(\alpha,\beta,\gamma)/2}}+\norm{\overline{B_0}-B^n_0}_{\mathbf{H}^{-1-\theta}}.    
\end{align}
Thanks to \autoref{convergence_operators}, \eqref{uniform bound matrix} and \eqref{uniform bound auxiliary A} it holds
\begin{align}\label{intermediate B step 2}
I_2 & \lesssim \frac{1}{n^{\theta \chi(\alpha,\beta,\gamma)/2}}\int_0^t\lvert t-s\rvert^{\theta/2}\norm{\nabla\times\left(\mathcal{A}^n_{\rho}\widehat{B^{n}_s}\right) ds}_{\mathbf{H}^{-1}} ds\notag\\ & \lesssim \frac{1}{n^{\theta \chi(\alpha,\beta,\gamma)/2}}\int_0^t\lvert t-s\rvert^{\theta/2}\norm{\widehat{B^{n}_s}}  ds \notag\\ & \lesssim_{T} \frac{1}{n^{\theta \chi(\alpha,\beta,\gamma)/2}}.    
\end{align}
Due to \autoref{Properties semigroup} and \eqref{uniform bound auxiliary A} we can treat $I_3$ obtaining easily:
\begin{align}\label{intermediate B step 3}
  I_3&\lesssim_{\eta,\theta} \int_0^t\norm{\nabla\times\left(\left(\mathcal{A}^n_{\rho}-\mathcal{A}_{\rho}\right)\widehat{B^{n}_s}\right)}_{\mathbf{H}^{-1}} ds \notag\\ & \lesssim \norm{\mathcal{A}^n_{\rho}-\mathcal{A}_{\rho}}_{HS}\int_0^T \norm{\widehat{B^{n}_s}} ds \notag\\ & \lesssim_{T}\frac{1}{n}.
\end{align}
Lastly we treat $I_4$ by \autoref{Properties semigroup}:
\begin{align}\label{intermediate B step 5}
 I_4 & \lesssim_{\eta} \int_0^t \frac{1}{(t-s)^{1/2}} \norm{\nabla\times\left(\mathcal{A}_{\rho}\left(\widehat{B^{n}_s}-\overline{B_s}\right)\right)}_{\mathbf{H}^{-2-\theta}}ds\notag\\ & \lesssim \int_0^t \frac{1}{(t-s)^{1/2}} \norm{\widehat{B^{n}_s}-\overline{B_s}}_{\mathbf{H}^{-1-\theta}}ds.
\end{align}
Combining \eqref{intermediate B step 1},\eqref{intermediate B step 2},\eqref{intermediate B step 3},\eqref{intermediate B step 5} we get \eqref{intermediate_convergence_B} by Gronwall's lemma. We move now to $\widehat{B^{n,3}_t}-\overline{B^{n,3}_t}$. First we recall that in case of \autoref{HP noise 2} $\chi(\alpha,\beta,\gamma)\equiv 1$ and we have 
\begin{align*}
\norm{\widehat{B^{n,3}_t}-\overline{B^3_t}}_{\Dot{H}^{-\vartheta}}& \lesssim \left(\norm{\left(S^n(t)-S(t)\right)\overline{B^3_0}}_{\Dot{H}^{-\vartheta}}+\norm{S^n(t)\left(B^{n,3}_0-\overline{B^3_0}\right)}_{\Dot{H}^{-\vartheta}}\right)\\& +\zeta^n_{H,2}\norm{\int_0^t \left(S^n(t-s)-S(t-s)\right) \div_H\left(\widehat{B^{n,H}_s}^{\perp}\right) ds}_{\Dot{H}^{-\vartheta}}  \\ & + \lvert \zeta^n_{H,2}-\zeta_{H,2}\rvert\norm{\int_0^t S(t-s)\div_H\left(\widehat{B^{n,H}_s}^{\perp}\right) ds}_{\Dot{H}^{-\vartheta}}\\ & + \zeta_{H,2}\norm{\int_0^t S(t-s)\div\left(\widehat{B^{n,H}_s}^{\perp}-\overline{B^H_s}^{\perp}\right) ds}_{\Dot{H}^{-\vartheta}}=J_1+J_2+J_3+J_4. 
\end{align*}
$J_1$ can be treated as in the previous case obtaining
\begin{align}\label{estimate_b3}
    J_1\lesssim \frac{1}{n^{\vartheta/2}}+\norm{\overline{B^3_0}-B^{n,3}_0}_{\Dot{H}^{-\vartheta}}.
\end{align}
Thanks to \autoref{properties_semigroup_splitting_extension}, \autoref{convergence_operators}, \autoref{Properties semigroup} and \eqref{uniform bound auxiliary A} it holds
\begin{align}\label{intermediate B3 step 2}
J_2 & \lesssim  \frac{1}{n^{{\vartheta/2}}}\int_0^t\lvert t-s\rvert^{\vartheta/2}\norm{e^{(t-s)\frac{\eta}{2}\Delta}\operatorname{div}_H\left(\widehat{B^{n,H}_s}\right)} ds\notag\\ & \lesssim_{\eta}  \frac{1}{n^{{\vartheta/2}}}\int_0^t \frac{1}{(t-s)^{1-\vartheta}}\norm{\widehat{B^{n}_s}}_{\mathbf{H}^{-1}} ds\notag\\ & \lesssim_{T} \frac{1}{\vartheta n^{{\vartheta/2}}}.    
\end{align}
Due to \autoref{Properties semigroup} we can treat $J_3$ obtaining:
\begin{align}\label{intermediate B3 step 3}
  J_3&\lesssim_{\eta,\vartheta}\frac{1}{n} \int_0^t \frac{1}{(t-s)^{1-\vartheta/2}} \norm{\operatorname{div}_H\left(\widehat{B^{n,H}_s}\right)}_{\Dot{H}^{-2}} ds \notag\\ & \lesssim \frac{1}{\vartheta n}.
\end{align}
Lastly we treat $J_4$ combining \autoref{Properties semigroup} and \eqref{intermediate_convergence_B}:
\begin{align}\label{intermediate B3 step 5}
 J_4 & \lesssim_{\eta} \int_0^t \frac{1}{(t-s)^{1+\theta/2-\vartheta/2}} \norm{\widehat{B^{n}_s}-\overline{B_s}}_{\Dot{H}^{-1-\theta}}ds\notag\\ & \lesssim_{T} \frac{1}{\vartheta-\theta}\sup_{t\in [0,T]}  \norm{\widehat{B^{n}_t}-\overline{B_t}}_{\Dot{H}^{-1-\theta}}\notag\\ & \lesssim_{\eta,T} \frac{ n^{-\theta/2}+\norm{\overline{B_0}-B^n_0}_{\mathbf{H}^{-1-\theta}}}{\vartheta-\theta} .
\end{align}
Combining \eqref{estimate_b3},\eqref{intermediate B3 step 2},\eqref{intermediate B3 step 3},\eqref{intermediate B3 step 5} we obtain relation \eqref{intermediate_convergence_B3}.
\end{proof}
Secondly we provide a quantitative result on the closeness of 
$B^n_t$ (resp. $B^n_t,B^{n,3}_t$) 
and $\widehat{B^{n}_t}$ (resp. $\widehat{B^{n}_t},\widehat{B^{n,3}_t}$). 
\begin{lemma}\label{preliminary convergence 2}
Under the same assumptions of \autoref{main Theorem}, for each  $\kappa \in [1,2),\ \theta\in (0,2),$ and $\delta\in (0,\theta)$ we have
\begin{align}\label{intermediate estimate 2 A}
\sup_{t\in [0,T]}\expt{\norm{\widehat{B^{n}_t}-B^n_t}^{\kappa}_{\mathbf{H}^{-1-\theta}}}&\lesssim  \frac{1}{n^{\frac{\kappa(\alpha\wedge\beta\wedge\gamma)(\theta-\delta)}{6+\delta}}}\norm{B^n_0}_{\mathbf{H}^{-1}}^\kappa
\end{align}
In case of \autoref{HP noise 2} we have also for each $\vartheta\in (\theta,\frac{3}{2}]$
\begin{align}
\sup_{t\in [0,T]}\expt{\norm{\widehat{B^{n,3}_t}-B^{n,3}_t}^\kappa_{\Dot{H}^{-\vartheta}}}&\lesssim  \frac{\norm{B^{n,3}_0}^\kappa}{n^{\frac{2\kappa(\vartheta-\delta)}{3}}}+ \frac{\norm{B^n_0}_{\mathbf{H}^{-1}}^\kappa }{n^{\frac{\kappa(\vartheta-\theta)(\theta-\delta)}{6+\delta}}}   . 
\end{align}
\end{lemma}
\begin{proof}
From \autoref{mild form}, we already know that $B^n_t$ and $B^{n,3}_t$ can be rewritten in mild form as 
\begin{align*}
   B^n_t&=S^n(t)B^n_0+\int_0^t S^n(t-s)\nabla \times (\mathcal{A}^n_{\rho}B^n_s)ds+ Z^{n}_t,\\
   B^{n,3}_t&=S^n(t)B^{n,3}_0-\frac{\rho \zeta^n_{H,2}} {C_{1,H}C_{2,H}}\int_0^t S^{n}(t-s)\operatorname{div}_H\left((B^{n,H}_s)^{\perp}\right) ds+H^{n}_t.
\end{align*}
Therefore, in case of \autoref{HP noise 1} and \autoref{HP noise isotropo} we introduce  
\begin{align*}
   d^{n}_t=B^n_t-\widehat{B^n_t},
\end{align*}
while in case of \autoref{HP noise 2} we introduce also 
\begin{align*}
     D^{n}_t=B^{n,3}_t-\widehat{B^{n,3}_t}.
\end{align*}
we have
\begin{align*}
    d^n_t=\int_0^t S^n(t-s)\nabla \times (\mathcal{A}^n_{\rho}d^n_s)ds+Z^n_t,\quad 
    D^{n}_t=H^n_t-\frac{\rho \zeta^n_{H,2}} {C_{1,H}C_{2,H}}\int_0^t S^{n}(t-s)\operatorname{div}_H\left((d^{n,H}_s)^{\perp}\right) ds.
\end{align*}    
Due to \autoref{cor_stoch_conv} and Gr\"onwall's inequality we can obtain \eqref{intermediate estimate 2 A}.
Indeed,
\begin{align*}
    \expt{\norm{d^n_t}_{\mathbf{H}^{-1-\theta}}^\kappa}&\lesssim_{\eta}\expt{\left(\int_0^t \frac{1}{\sqrt{t-s}}\norm{d^n_s}_{\mathbf{H}^{-1-\theta}}ds\right)^\kappa}+\expt{\norm{Z^n_t}_{\mathbf{H}^{-1-\theta}}^2}^{\frac{\kappa}{2}}\\ & \lesssim_T \int_0^t \frac{1}{(t-s)^{\kappa/2}}\expt{\norm{d^n_s}^{\kappa}_{\mathbf{H}^{-1-\theta}}}ds+\expt{\norm{Z^n_t}_{\mathbf{H}^{-1-\theta}}^2}^{\frac{\kappa}{2}}.
\end{align*}
This relation easily implies \eqref{intermediate estimate 2 A}.
Now we move to the analysis of $D^{n}_t$. $H^{n}_t$ can be treated easily by \autoref{cor_stoch_conv} obtaining 
\begin{align}\label{intermediate D step 1}
\sup_{t\in [0,T]}\expt{\norm{H^{n}_t}^\kappa_{\Dot{H}^{-\vartheta}}}&\lesssim   \frac{\norm{B^{n,3}_0}^\kappa+\norm{B^n_0}_{\mathbf{H}^{-1}}^\kappa}{ n^{\frac{2\kappa(\vartheta-\delta)}{3}}}.    
\end{align}
In order to treat the deterministic convolution, we first argue by interpolation obtaining thanks to \autoref{Properties semigroup} \begin{align}\label{intermediate D step 2}
 &\norm{\int_0^t S^{n}(t-s)\operatorname{div}_H\left((d^{n,H}_s)^{\perp}\right) ds}^\kappa_{\Dot{H}^{-\vartheta}}\notag\\ &\leq \norm{\int_0^t S^{n}(t-s)\operatorname{div}_H\left((d^{n,H}_s)^{\perp}\right) ds}^{\frac{\kappa(\vartheta-\theta)}{2}}_{\Dot{H}^{-2-\theta}}   \norm{\int_0^t S^{n}(t-s)\operatorname{div}_H\left((d^{n,H}_s)^{\perp}\right) ds}^{\frac{\kappa(2-\vartheta+\theta)}{2}}_{\Dot{H}^{-\theta}}\notag\\ &   \lesssim_{\eta} \left(\int_0^t \norm{d^{n}_s}_{\mathbf{H}^{-1-\theta}}ds \right)^{\frac{\kappa(\vartheta-\theta)}{2}}\left(\int_0^t \frac{1}{(t-s)^{1-\theta/2}} \norm{\widehat{B^{n}_s}-B^n_{s}}_{\mathbf{H}^{-1}} ds\right)^{\frac{\kappa(2-\vartheta+\theta)}{2}}\notag\\ & \lesssim_{T} \frac{1}{\theta^{\frac{\kappa(2-\vartheta+\theta)}{2}}}\left(\int_0^t \norm{d^{n}_s}_{\mathbf{H}^{-1-\theta}}^\kappa ds \right)^{\frac{\vartheta-\theta}{2}}\sup_{t\in [0,T]}\left(\norm{B^n_t}_{\mathbf{H}^{-1}}+\norm{\widehat{B^{n}_t}}_{\mathbf{H}^{-1}}\right)^{\frac{\kappa(2-\vartheta+\theta)}{2}}.
\end{align}
Combining \eqref{intermediate D step 2}, \eqref{intermediate estimate 2 A}, \eqref{prop:compact_space}, \eqref{uniform bound auxiliary A} we obtain an estimate of the deterministic convolution. Indeed, by H\"older inequality, for each $t\in [0,T]$ it holds:
\begin{align}\label{intermediate D step 3}
&\left(\frac{\lvert\rho\rvert \zeta^n_{H,2}} {C_{1,H}C_{2,H}}\right)^{\kappa}\expt{\norm{\int_0^t S^{n}(t-s)\operatorname{div}_H\left((d^{n,H}_s)^{\perp}\right) ds}^\kappa_{\Dot{H}^{-\vartheta}}} \notag\\ & \lesssim_{\eta,T,\theta}\expt{\left(\int_0^t \norm{d^{n}_s}_{\mathbf{H}^{-1-\theta}}^\kappa ds \right)^{\frac{\vartheta-\theta}{2}}\sup_{t\in [0,T]}\left(\norm{B^n_t}_{\mathbf{H}^{-1}}+\norm{\widehat{B^{n}_t}}_{\mathbf{H}^{-1}}\right)^{\frac{\kappa(2-\vartheta+\theta)}{2}}}\notag\\ & \leq  \expt{\int_0^T \norm{d^{n}_s}_{\mathbf{H}^{-1-\theta}}^\kappa ds }^{\frac{\vartheta-\theta}{2}}\expt{\sup_{t\in [0,T]}\left(\norm{B^n_t}_{\mathbf{H}^{-1}}^\kappa+\norm{\widehat{B^{n}_t}}_{\mathbf{H}^{-1}}^\kappa\right)}^{\frac{2-\vartheta+\theta}{2}}\notag \\ & \lesssim_{\delta}   \frac{1}{n^{\frac{\kappa(\vartheta-\theta)(\theta-\delta)}{6+\delta}}}  \norm{B^n_0}_{\mathbf{H}^{-1}}^\kappa.
\end{align}
Combining \eqref{intermediate D step 1} and \eqref{intermediate D step 3} the result follows.
\end{proof}

The proof of \autoref{main Theorem} then follows combining \autoref{preliminary convergence} and \autoref{preliminary convergence 2}.

%%%%%%%%%%%%%%%%%%%%%%%%%%%%%%%%%%%%%%%%%%%%%%%%%%%%%%%%%%%%%%%%%%%%%%%%%%%%%%%%%%%%%%%%%%%%%
\begin{acknowledgements}
The authors thank Professor Franco Flandoli and Professor Francesco Pegoraro for useful discussions and valuable insight into the subject.   
\end{acknowledgements}
\appendix  
\section{Some useful lemmas}
\subsection{Proof of \autoref{convergence_operators}}\label{app_convergence_operator}
    Thanks to the Fourier expansion formula \eqref{fourier_description_op} we have
    \begin{align*}
    \lVert \left( S^{\mu_1,\nu_1}(t)-S^{\mu_2,\nu_2}(t)\right)q\rVert^2_{\Dot{H}^{s-\varphi}}&=\sum_{k\in \Z^3_0}\left(e^{-\left(\mu_1(k_1^2+k_2^2)+\nu_1 k_3^2\right)t}-e^{-\left(\mu_2(k_1^2+k_2^2)+\nu_2 k_3^2\right)t}\right)^2\langle q,\frac{e^{-ik\cdot x}}{(2\pi)^{3/2}}\rangle^2 \frac{\lvert k\rvert^{2s}}{\lvert k\rvert^{2\varphi}}\\ & \leq \norm{q}_{\Dot{H}^{s}}^2\sup_{k\in \Z^3_0}\frac{\left(e^{-\left(\mu_1(k_1^2+k_2^2)+\nu_1 k_3^2\right)t}-e^{-\left(\mu_2(k_1^2+k_2^2)+\nu_2 k_3^2\right)t}\right)^2}{\lvert k\rvert^{2\varphi}}.   
    \end{align*}
    Therefore, the proof is complete if we are able to show that 
    \begin{align*}
    \sup_{k\in \Z^3_0}\frac{\left(e^{-\left(\mu_1(k_1^2+k_2^2)+\nu_1 k_3^2\right)t}-e^{-\left(\mu_2(k_1^2+k_2^2)+\nu_2 k_3^2\right)t}\right)^2}{\lvert k\rvert^{2\varphi}}\lesssim_{\varphi}  \left(\lvert \mu_1-\mu_2\rvert+\lvert \nu_1-\nu_2\rvert\right)^{\varphi}t^{\varphi}. 
    \end{align*}
    This is indeed the case. If $\mu_1<\mu_2$ and $\nu_1<\nu_2$
    \begin{align*}
        \frac{\left(e^{-\left(\mu_1(k_1^2+k_2^2)+\nu_1 k_3^2\right)t}-e^{-\left(\mu_2(k_1^2+k_2^2)+\nu_2 k_3^2\right)t}\right)^2}{\lvert k\rvert^{2\varphi}}&\leq \frac{\left(1-e^{-\left((\mu_2-\mu_1)(k_1^2+k_2^2)+(\nu_2-\nu_1) k_3^2\right)t}\right)^2}{\lvert k\rvert^{2\varphi}}\\ & \leq t^{\varphi}(\lvert \mu_2-\mu_1\rvert\wedge \lvert \nu_2-\nu_1\rvert )^{\varphi}\frac{\left(1-e^{-\left((\mu_2-\mu_1)(k_1^2+k_2^2)+(\nu_2-\nu_1) k_3^2\right)t}\right)^2}{t^{\varphi}\left( \lvert \mu_1-\mu_2\rvert (k_1^2+k_2^2)+\lvert \nu_1-\nu_2\rvert k_3^2\right)^{\varphi}}\\ & \lesssim t^{\varphi}(\lvert \mu_2-\mu_1\rvert\wedge \lvert \nu_2-\nu_1\rvert )^{\varphi}
    \end{align*}
     since $\frac{1-e^{-x}}{x^{\varphi/2}}$ is bounded for $x> 0$, uniformly in $\varphi\in [0,2]$. In case of $\mu_1<\mu_2$ and $\nu_2<\nu_1$ and either $(k_1,k_2)\neq 0$ and $k_3\neq 0$ 
         \begin{align*}
        \frac{\left(e^{-\left(\mu_1(k_1^2+k_2^2)+\nu_1 k_3^2\right)t}-e^{-\left(\mu_2(k_1^2+k_2^2)+\nu_2 k_3^2\right)t}\right)^2}{\lvert k\rvert^{2\varphi}}&\lesssim \frac{\left(1-e^{-\left((\mu_2-\mu_1)(k_1^2+k_2^2)\right)t}\right)^2}{ \left(k_1^2+k_2^2\right)^{\varphi}}+\frac{\left(1-e^{-\left((\nu_1-\nu_2)k_3^2\right)t}\right)^2}{ \lvert k_3\rvert^{2\varphi}}\\ & \leq t^{\varphi}\lvert \mu_2-\mu_1\rvert^{\varphi}\frac{\left(1-e^{-(\mu_2-\mu_1)(k_1^2+k_2^2)t}\right)^2}{\left( t\lvert \mu_1-\mu_2\rvert (k_1^2+k_2^2)\right)^{\varphi}}\\ & +t^{\varphi}\lvert \nu_2-\nu_1\rvert^{\varphi}\frac{\left(1-e^{-(\nu_1-\nu_2)k_3^2t}\right)^2}{\left( t\lvert \nu_1-\nu_2\rvert k_3^2\right)^{\varphi}}\\ & \lesssim t^{\varphi}\left(\lvert \mu_2-\mu_1\rvert^{\varphi}+\lvert \nu_2-\nu_1\rvert^{\varphi}\right).
    \end{align*}
    The case $(k_1,k_2)=0$ or $k_3=0$ easily reconducts to the one treated. This completes the proof.
\subsection{Proof of \autoref{lemma ito strat corrector}}\label{app_ito_strat}
    All the terms involved are Lie brackets, thus we begin with some preliminary computations. By recalling that we have $\nabla\skj = i \tkj k\otimes a_{k, j}e^{ikx}$, and also that $a_{-k, j}=a_{k,j}$ for all $k \in \Z_0^3, \  j=1,2$ we write 
\begin{align} \label{TT}
    \skl\cdot \nabla(\smkj \cdot \nabla F) &= (\skl\cdot \nabla \smkj)\cdot \nabla F + \skl\otimes \smkj : \nabla^2F  \\
    &= -i\tkl(a_{k,l}\cdot k)a_{-k,l}\cdot \nabla F + \skl\otimes \smkj : \nabla^2F \notag \\
    &= \skl\otimes \smkj : \nabla^2F \notag
\end{align}
since $a_{k,l}\cdot k= 0$.
On the other hand we have 
\begin{align} \label{SS}
    (F\cdot\nabla\smkj)\cdot\nabla\skl = \tmkj\tkl (F \cdot k)(a_{k,j}\cdot k)a_{k,l} = 0
\end{align}
while 
\begin{align} \label{TS}
    \smkj\cdot\nabla\left(F\cdot\nabla\skl\right)&= 
   i\tkl\tmkj e^{-ikx} a_{k,j}^r \partial_r ((F\cdot k) a_{k, l} e^{ikx})\\ &= i\tkl\tmkj e^{-ikx} a_{k,j}^r((\partial_r F\cdot k) a_{k, l} e^{ikx} + (F\cdot k) a_{k, l} k_r e^{ikx}) \notag\\ &= i\tkl\tmkj a_{k,j}^r(\partial_r F\cdot k) a_{k, l} \notag
\end{align}
and 
\begin{align} \label{ST}
    \left(\skl\cdot\nabla F\right)\cdot\nabla\smkj &=-i\tkl\tmkj a^r_{k,l}(\partial_r F\cdot k ) a_{k,j}.
\end{align}
In particular we see that if $l=j$ the sum of \eqref{TS} and \eqref{ST} is zero.  \\
We are now ready to compute the correctors. We have
\begin{align*}
     \mathcal{L}_{\skj}\mathcal{L}_{\smkj} F&=  \skj\cdot \grad(\smkj\cdot \nabla F) -\skj\cdot \grad(F\cdot \nabla \smkj) - (\smkj\cdot \nabla F)\cdot\nabla \skj   +(F\cdot \nabla \smkj)\cdot\nabla\skj.  
\end{align*}
By \eqref{SS} the last term is zero, while by \eqref{TS} and \eqref{ST}, the sum of the two middle terms is zero. Summing over $k\in \Z_0^3$ and $j=1,2$ we get 
\begin{align}\label{TTcov}
    \sumkj  \mathcal{L}_{\skj}\mathcal{L}_{\smkj} F &= \sumkj  \skl\otimes \smkj : \nabla^2F \\
    &= \frac{1}{2}\div(Q_0^n(0)\grad F) \notag \\
    &= \Lambda^n F \notag
\end{align}
where we used again the fact that $\skj\cdot \grad \smkj = 0$ in the second line and \autoref{form covariance matrix non-isotropic}, \autoref{form covariance matrix isotropic} in the last equality. \\
We now look at the part of the corrector relative to the correlation $\rho$, that is 
\begin{align}
   \rho \sumkmeancov \left(\mathcal{L}_{\sk{1}}\mathcal{L}_{\smk{2}}+\mathcal{L}_{\sk{2}}\mathcal{L}_{\smk{1}}\right)
\end{align}
We write the general form of the two summands, for $l\neq j$, as
\begin{align*}
     \mathcal{L}_{\skl}\mathcal{L}_{\smkj} F&=  \skl\cdot \grad(\smkj\cdot \nabla F) -\skl\cdot \grad(F\cdot \nabla \smkj) - (\smkj\cdot \nabla F)\cdot\nabla \skl   +(F\cdot \nabla \smkj)\cdot\nabla\skl.  
\end{align*}
Again, by \eqref{SS} the last term is always zero. Now summing over $k$ and $l\neq j$ we see, using first \eqref{TT} and reasoning as in the manipulation of \eqref{TTcov}
\begin{align}
    \sumkmeancov\sigma^n_{k,1}\cdot \grad(\sigma^n_{-k,2}\cdot \nabla F) +  \sigma^n_{k,2}\cdot \grad(\sigma^n_{-k,1}\cdot \nabla F) &= \div(\overline{Q_\rho^n}(0)\nabla F)= 0.
\end{align}
We are left with the contributions of the middle terms which no longer sum to zero as in the first corrector. Recalling our choice of the coefficients \eqref{definition thetakj}, we see that if $l\neq j$ and $n\leq \lvert k\rvert \leq 2n,\ k_3=0$
\begin{equation*}
    \tkj\tmkl  = - \tkl \tmkj.
\end{equation*}
Thus, when we sum over $l\neq j$, we get from \eqref{TS} and \eqref{ST} that 
\begin{align}
    -\sigma^n_{k, 1}\cdot \grad(F\cdot \nabla \sigma^n_{-k, 2}) - (\sigma^n_{-k, 1}\cdot \nabla F)\cdot\nabla \sigma^n_{k, 2} & \\
    &= 2 i\theta_{k, 1}\theta_{-k, 2} a^r_{k,1}(\partial_r F\cdot k ) a_{k,2} \notag \\
    &= -\frac{2\sgn(k)}{|k|^{\alpha/2 + \beta/2}} a^r_{k,1}(\partial_r F\cdot k ) a_{k,2} \notag
\end{align}
and similarly
\begin{align}
     - (\sigma^n_{-k, 2}\cdot \nabla F)\cdot\nabla \sigma^n_{k, 1} -\sigma^n_{k, 2}\cdot \grad(F\cdot \nabla \sigma^n_{-k, 1}) & \\
    &=- 2 i\theta^n_{-k,2}\theta^n_{k, 1} a^r_{k,2}(\partial_r F\cdot k ) a_{k,1} \notag \\
    &= \frac{2\sgn(k)}{|k|^{\alpha/2 + \beta/2}}a^r_{k,2}(\partial_r F\cdot k ) a_{k,1} \notag 
\end{align}
where we indicated with $\sgn(k)$ the function that takes value $1$ if $k\in \Gamma^+$ and $-1$ if $k\in \Gamma^-$. Summing the last two expressions and then summing over $k\in \Z_0^3, \ k_3=0$ it is straightforward to see thanks to \autoref{convergence properties matrix} that 
\begin{align*}
   \rho \sumkmeancov \left(\mathcal{L}_{\sk{1}}\mathcal{L}_{\smk{2}}+\mathcal{L}_{\sk{2}}\mathcal{L}_{\smk{1}}\right) &= 
   -\sum_{l\in \{1,2\}} \partial_l \overline{Q^{n}_{\rho}}(0)\cdot\nabla F_l \\
   &= \Lambda^n_\rho F.
\end{align*}

\bibliography{main}{}
\bibliographystyle{plain}

\end{document}